\documentclass{amsart}
\usepackage{blkarray}
\usepackage{amsmath}
\usepackage{amsfonts}
\usepackage{amsthm}
\usepackage{hyperref}
\usepackage{pstricks}

\newcommand{\A}{\mathbb{A}}
\newcommand{\bF}{\mathbf{F}}

\newcommand{\KK}{\mathbb{K}}
\newcommand{\PP}{\mathbb{P}}

\newcommand{\U}{\mathcal{U}}
\newcommand{\F}{\mathcal{F}}
\newcommand{\J}{\mathcal{J}}
\newcommand{\mcL}{\mathcal{L}}
\newcommand{\CO}{\mathcal{O}}
\newcommand{\fd}{\mathbf{F}_k(\dmn)}
\newcommand{\fp}{\mathbf{F}_k(\pmn)}
\newcommand{\dmn}{D_{m,n}^r}
\newcommand{\pmn}{P_{m,n}^r}

\DeclareMathOperator{\codim}{codim}
\DeclareMathOperator{\rank}{rank}
\DeclareMathOperator{\perm}{perm}
\DeclareMathOperator{\coeff}{coeff}
\DeclareMathOperator{\chara}{char}
\DeclareMathOperator{\Hom}{Hom}
\DeclareMathOperator{\Gr}{Gr}
\DeclareMathOperator{\Spec}{Spec}
\DeclareMathOperator{\Sym}{Sym}
\DeclareMathOperator{\GL}{GL}
\newcommand{\comp}{\mathfrak{C}}
\newcommand{\comment}[1]{}
\usepackage{enumerate}
\newtheorem{theorem}{Theorem}[section]
\newtheorem{prop}[theorem]{Proposition}
\newtheorem{cor}[theorem]{Corollary}

\newtheorem{conj}[theorem]{Conjecture}

\newtheorem{lemma}[theorem]{Lemma}
\theoremstyle{definition}

\newtheorem{defn}[theorem]{Definition}
\newtheorem{rem}[theorem]{Remark}
\newtheorem{question}[theorem]{Question}
\newtheorem{ex}[theorem]{Example}

\title{Fano Schemes of Determinants and Permanents}

\author{Melody Chan}
\address{Department of Mathematics, Harvard University
Cambridge, MA 02138, USA}
\email{mtchan@math.harvard.edu}

\author{Nathan Ilten}
\address{Department of Mathematics, Simon Fraser University,
8888 University Drive, Burnaby BC V5A1S6, Canada }
\email{nilten@sfu.ca}

\newcommand{\intersectioncompa}{
  \psset{unit=2cm}
  \begin{pspicture}(0,0)(3,3)
    \psdots[dotstyle=o](0,0)(1,0)(2,0)(3,0)(3,1)(3,2)(3,3)
\psdots(0,1)(1,1)(2,1)(0,2)(1,2)(2,2)(0,3)(1,3)(2,3)
\rput(1.5,1.5){
  {\tiny{$\left(\begin{array}{c c c}
0&*&0\\
*&*&*\\
0&*&0
  \end{array}\right)$
}}}
\rput(0.5,1.5){
  {\tiny{$\left(\begin{array}{c c c}
*&0&0\\
*&*&*\\
*&0&0
  \end{array}\right)$
}}}
\rput(2.5,1.5){
  {\tiny{$\left(\begin{array}{c c c}
0&0&*\\
*&*&*\\
0&0&*
  \end{array}\right)$
}}}
\rput(1.5,0.5){
  {\tiny{$\left(\begin{array}{c c c}
0&*&0\\
0&*&0\\
*&*&*
  \end{array}\right)$
}}}
\rput(1.5,2.5){
  {\tiny{$\left(\begin{array}{c c c}
*&*&*\\
0&*&0\\
0&*&0
  \end{array}\right)$
}}}
\rput(0.5,0.5){
  {\tiny{$\left(\begin{array}{c c c}
*&0&0\\
*&0&0\\
*&*&*
  \end{array}\right)$
}}}
\rput(0.5,2.5){
  {\tiny{$\left(\begin{array}{c c c}
*&*&*\\
*&0&0\\
*&0&0
  \end{array}\right)$
}}}
\rput(2.5,2.5){
  {\tiny{$\left(\begin{array}{c c c}
*&*&*\\
0&0&*\\
0&0&*
  \end{array}\right)$
}}}
\rput(2.5,0.5){
  {\tiny{$\left(\begin{array}{c c c}
0&0&*\\
0&0&*\\
*&*&*
  \end{array}\right)$
}}}
    \psline(3,3)(0,3)(0,0)    
    \psline[linestyle=dashed](0,0)(3,0)(3,3)
    \psline(1,0)(1,3)
    \psline(2,0)(2,3)
    \psline(0,1)(3,1)
    \psline(0,2)(3,2)
\end{pspicture}
}

\begin{document}
\maketitle
\begin{abstract}
Let $\dmn$ and $\pmn$ denote the subschemes of $\PP^{mn-1}$ given by the $r\times r$ determinants (respectively the $r\times r$ permanents) of an $m\times n$ matrix of indeterminates.  In this paper, we study the geometry of the Fano schemes $\fd$ and $\fp$ parametrizing the $k$-dimensional planes in $\PP^{mn-1}$ lying on $\dmn$ and $\pmn$, respectively.  We prove results characterizing which of these Fano schemes are smooth, irreducible, and connected; and we give examples showing that they need not be reduced.  We show that $\bF_1(D_{n,n}^n)$ always has the expected dimension, and we describe its components exactly.  Finally, we give a detailed study of the Fano schemes of $k$-planes on the $3\times 3$ determinantal and permanental hypersurfaces.
\end{abstract}

  \section{Introduction}\label{sec:intro}
Fix an algebraically closed field $\KK$. For numbers $r,m,n$ with $1<r\leq m\leq n$, let $\dmn$ and $\pmn$ denote the subschemes of $\PP_\KK^{mn-1}$ defined by the $r\times r$ determinants, respectively the $r\times r$ permanents, of an $m\times n$ matrix
$$\left(\!
\begin{array}{c c c}
	x_{1,1}& \cdots & x_{1,n}\\
	\vdots & \ddots& \vdots\\
	x_{m,1}&\cdots &x_{m,n}
\end{array}
\!\right)
$$
filled with $m n$ independent forms $x_{i,j}$. Whenever we are dealing with $\pmn$, we will make the standard assumption that $\chara \KK\neq 2$.
In this article, we study the Fano schemes $\bF_k(\dmn)$ and $\bF_k(\pmn)$.  These are subschemes of the Grassmannian $\Gr(k+1, mn)$  
parametrizing those $k$-dimensional planes in $\PP_\KK^{mn-1}$ that are contained in the schemes $\dmn$ and $\pmn$, respectively.

We have three main reasons for studying these Fano schemes.  
First, we would like to understand general Hilbert schemes better.  In the case of classical Hilbert schemes, i.e.~those parametrizing subschemes of $\PP^n$, well-understood examples are scarce, but we do know a number of general results: for example, classical Hilbert schemes are always connected \cite{hartshorne:66a}. 
In the case of general Hilbert schemes, i.e.~those parametrizing subschemes of a fixed closed subscheme $X\subset \PP^n$ that have some fixed Hilbert polynomial, even less is known.  Our Fano schemes $\bF_k(\dmn)$ and $\bF_k(\pmn)$ provide an interesting family  of Hilbert schemes whose study is tractable but whose geometry is still very rich.  

Our second reason for studying $\bF_k(\dmn)$ and $\bF_k(\pmn)$ comes from geometric complexity theory.  It is well-known that permanents and determinants behave completely differently from the perspective of complexity theory.  Indeed, computing the permanent of a square matrix is $\#P$-hard  \cite{valiant:79a}, while the determinant is computable in polynomial time.  In fact, one of the central conjectures in complexity theory, due to Valiant, posits that the permanent of an $n\times n$ matrix cannot be computed by affine-linear substitution from the determinant of square matrix whose size is polynomial in $n$ \cite{valiant:79a}.  Recently, Mulmuley and Sohoni developed an interesting representation-theoretic approach to this conjecture which is being pursued by a number of authors, see e.g.~\cite{gct} and the references there.  This new approach suggests that it is worthwhile to revisit the determinant and permanent from a geometric perspective.  More specifically, linear spaces lying in the determinantal and permanental hypersurfaces in $\PP^{n^2-1}$ play a particularly important role \cite[\S 5]{landsberg}.  Indeed, the spaces $\bF_k(\dmn)$ and $\bF_k(\pmn)$ do exhibit a number of interesting differences, as we will see.  (Actually, we began our study with the case $r=m=n$, which is the most interesting case for complexity theorists.  We then realized that we could extend our techniques to $\dmn$ and $\pmn$, and that their Fano schemes have some interesting geometric features that cannot be seen just in the hypersurface case.)

Third, we are interested in contributing to the study of permanental ideals, i.e.~the ideals defined by the $r\times r$ permanents of an $m\times n$ matrix.  These ideals are much less well-studied and behave very differently than determinantal ideals.
 For example, they can have many primary components, including embedded components, and in general their primary decompositions are not known, with a few nice exceptions, e.g.~\cite{kirkup, laubenbacher}.  In our paper, we are able to get new information about the linear spaces in these permanental schemes without computing their primary decompositions at all.

We next summarize the main results in this paper, starting with $\dmn$ in \S\ref{subsec:det} and turning to $\pmn$ in \S\ref{subsec:perm}.

\subsection{The Fano scheme $\fd$}\label{subsec:det}
First, let us define two quantities which appear in our main theorems. Fix $r,m,$ and $n$, and for $0\leq s\leq r-1$, let
\begin{align}
    \kappa(s) &= mn-(m-s)(s+n-r+1)-1, \text{ and}\label{eqn:kappa}\\
    \delta(s)&=(s+1+n-r)(r-s-1)+s(m-s).\label{eqn:delta}
\end{align}
The interpretation of $\kappa(s)$ is that it is the dimension of the linear space in $\PP^{mn-1}$ of maps $\KK^n\rightarrow \KK^m$ sending a fixed $(s+n-r+1)$-dimensional subspace $V$ of $\KK^n$ into a fixed $s$-dimensional subspace $W$ of $\KK^m$.  So, for example, $\kappa(s)+1$ is the dimension of the space of matrices with a zero block as shown below in (\ref{fig:compression}).   These linear spaces were studied first in \cite{eisenbud:88a} and are very important to our analysis; see Definition~\ref{defn:compression}.  Next, $\delta(s)$ is simply the dimension of the product of Grassmannians $\Gr(s+n-r+1,n)\times\Gr(s,m)$ that parametrizes our choices of $V$ and $W$.

\begin{equation}\label{fig:compression}
\bordermatrix{ & s\!+\!1\!+\!n\!-\!r & r\!-\!s\!-\!1 \cr
& 0\enskip0\enskip0\enskip0\enskip0 & \qquad\cr
m\!-\!s & 0\enskip0\enskip0\enskip0\enskip0 & \cr
& 0\enskip0\enskip0\enskip0\enskip0 & \cr
& & \cr
\quad s& & \cr
& & 
}\end{equation}
\\

Our first theorem on Fano schemes of determinants is a natural generalization of \cite[Corollary 2.2]{eisenbud:88a} and gives the complete picture in the case $k=1$ of the Fano scheme of lines.
\begin{theorem}\label{thm:lines}
  The Fano scheme $\bF_1(\dmn)$ has  exactly $r$ irreducible components, of dimensions
  $$
\delta(s)+2(\kappa(s)-1) \qquad 0\leq s \leq r-1.
  $$
  In particular if $m=n$, then all irreducible components of $\bF_1(D_{n,n}^r)$ have dimension $(n-r)(r-2)+2nr-n-5.$  If $r>2$, then all components intersect pairwise.  Furthermore, if $m=n=r$, then $\bF_1(D_{n,n}^n)$ has the expected dimension and is a reduced local complete intersection.
\end{theorem}
\noindent We prove this theorem in Section~\ref{sec:lines}.  We also explicitly calculate the degree of $\bF_1(D_{n,n}^n)$, as a subscheme of $\Gr(2,n^2)$ in its Pl\"ucker embedding, for $n$ up to 6. 
See Proposition \ref{prop:deglines}.

For higher values of $k$, we do not have a complete characterization of $\fd$, but we can still say a lot about these schemes.
The first thing we should note is that the Fano scheme $\bF_k(\dmn)$ is nonempty if and only if $k < (r-1)n$:  this follows from a result of Dieudonn\'e \cite{dieudonne} (see also \S\ref{sec:torus} for a quick proof).  For these values of $k$, we can say exactly which of the schemes $\fd$ are smooth and irreducible.  

\begin{theorem}\label{thm:smooth}
  Let $1\leq k <(r-1)n$. 
  \begin{enumerate}[(a)]
\item The Fano scheme $\bF_k(\dmn)$ is smooth if and only if $k>(r-2)n.$ 

\item The Fano scheme $\bF_k(\dmn)$ is irreducible if and only if $m\neq n$ and \newline $k>(r-2)n+m-r+1$. 
\end{enumerate}
\end{theorem}

\noindent See \S\ref{subsec:smooth1} for the proof.

Our next main result says exactly when when $\bF_k(D^m_{m,n})$ is connected.
\begin{theorem}\label{thm:connected}
  Suppose that $1\leq k <(m-1)n$. 
  Then $\bF_k(D_{m,n}^m)$ is disconnected if and only if
$$m^2-2m < k \leq \kappa(0)$$ 
\noindent	or if there exists an integer $s$ with $0<s<m-1$ such that 
 $$\kappa(s)-\min\{m\!-\!s\!-\!1,n\!-\!m\!+\!s\} \,<\,    k    \,\leq\, \kappa(s). $$
  \end{theorem}
  
\noindent See \S\ref{subsec:smooth1} for the proof.  Table~\ref{table:det} illustrates the results of Theorems~\ref{thm:smooth} and~\ref{thm:connected} for $\bF_k(D_{n,n}^n)$ for $n\leq 10$.

\begin{table}
  \begin{center}
    \begin{tabular}{c c c c c c c c c c  }  
\hline
      $n$& $2$& $3$& $4$& $5$& $6$& $7$ & $8$ & $9$ & $10$\\
\hline
\noalign{\vskip .5mm}
      Non-empty iff $k\leq$&  $1$&$5$&$11$&$19$&$29$&$41$&$55$&$71$&$89$   \\
\noalign{\vskip .5mm}
 \hline
 \noalign{\vskip .5mm}
      Singular iff $k\leq$& -- &$3$& $8$& $15$& $24$&$35$&$48$ &$63$ & $80$    \\
\noalign{\vskip .5mm}
\hline
\noalign{\vskip .5mm}
      Connected iff $k\leq$&   -- &$3$&$8$&$13$&$21$&$29$&$40$  &$51$& $65$       \\
      or $k=$& &&&&$24$&$35$&$46$--$48$&$57,60$--$63$&$72,73,76$--$80$\\
      \noalign{\vskip .5mm}
      \hline
    \end{tabular}

\vspace{.5cm}

  \end{center}\caption{Properties of the Fano scheme $\bF_k(D_{n,n}^n)$  for $n\leq 10$}\label{table:det}
\end{table}

Theorem~\ref{thm:connected} is very surprising.  It says that connectivity of $\bF_k(D_{m,n}^m)$  can actually be highly non-monotonic as $k$ varies from $1$ to $(m-1)n-1$, since each of the $m-2$ 
values of $s$ above cuts out an interval of values of $k$ for which $\bF_k(D_{m,n}^m)$ is disconnected.  The last row of Table~\ref{table:det} gives examples of this behavior.  
That said, $\bF_k(D_{m,n}^m)$ is always connected if $k$ is sufficiently small, as we show in  Corollary~\ref{cor:connected}.

As for connectedness of $\fd$ when $r<m$, we can still give some necessary conditions and sufficient conditions, although they are not strong enough to completely characterize when $\fd$ is connected.   See Theorem \ref{thm:connecteddet}.

In the next section, we will define a family of special irreducible components $\comp_k(s)$ of $\fd$ that we  call {\em compression components}.  These components are very important to our analysis of $\fd$.  Indeed, we will see that if $k=1$ or if $k$ is sufficiently large then $\fd$ consists {\em only} of compression components (Theorem~\ref{thm:lines} and Corollary~\ref{cor:disjoint} respectively). 

For general $k$, however, other components may appear, as detected in \cite[Theorem 1.1]{eisenbud:88a}.  For example, we will see in Section~\ref{sec:3} that a non-compression component already appears in $\bF_2(D^3_{3,3})$: it is the component $\comp^*$ of $2$-planes of matrices $(\GL_3\times \GL_3)$-equivalent to the space of $3\times 3$ antisymmetric matrices.  On the other hand, the compression components still have the desirable property that  
every component of $\fd$ must meet one of them (Remark~\ref{rem:intersect}).  This fact makes it very useful to study local neighborhoods of points on these components.

\subsection{The Fano scheme $\bF_k(\pmn)$}\label{subsec:perm}
The study of 
$\bF_k(\pmn)$ is slightly more delicate due to the lack of $\GL_m\times \GL_n$ symmetry. Nonetheless, we can prove several results on $\bF_k(\pmn)$ that have a similar flavor to those for $\bF_k(\dmn)$.  (When we discuss $\pmn$, we will always assume that $\chara (\KK)\neq 2$, since otherwise $\pmn=\dmn$ and our above results apply.)

For example, just as in the determinantal case, $\bF_k(\pmn)$ is nonempty if and only if $k<(r-1)n$ (Proposition~\ref{prop:nonempty}). 
Furthermore, in that range, we can completely characterize when  $\bF_k(\pmn)$ is smooth:
\begin{theorem}\label{thm:smoothp}
  Let $1\le k<(r-1)n$.
The Fano scheme $\bF_k(\pmn)$ is smooth if and only if $n=2$ or if $$k> (r-2)n+1.$$ 
\end{theorem}

\noindent However, if $\bF_k(\pmn)$ is nonempty, it is {\em never} irreducible, as we prove in Proposition \ref{prop:irredperm}.   See Table~\ref{table:perm} for a summary of our results applied to the case of the $P^n_{n,n}$ for $n\le 9$.  

We can also give necessary conditions and sufficient conditions for the connectedness of $\fp$: see Theorem \ref{thm:connectedp}. In particular, we show that if $k$ is sufficiently small, then $\fp$ is connected (Corollary~\ref{cor:connectedp}). However, we do not know how to completely characterize when $\bF_k(\pmn)$ is connected, even in the case  $r=m$.

 As in the case of the determinant, there is a family of {\em compression subschemes} $\comp_k(\sigma,\tau)$ of $\fp$, defined in Section~\ref{sec:torus}, that again play an important role in our analysis.  For example, any irreducible component of $\bF_k(\pmn)$ must intersect a subscheme of the form $\comp_k(\sigma,\tau)$: see Remark \ref{rem:intersect}.  We prove in Proposition~\ref{prop:perm} that these subschemes are (with a few necessary exceptions) actually components; and we show (Corollary~\ref{cor:permk}) that if $k$ is large enough then $\fp$ is simply a disjoint union of 
 compression components.

 In general, however, $\fp$ has more components than just those of the form $\comp_k(\sigma,\tau)$. In Section~\ref{sec:threetimesthreep}, we consider in depth the example of $\bF_k(P_{3,3}^3)$ for various values of $k$. In particular, we give a complete description of $\bF_4(P^3_{3,3})$: it consists of nine $\PP^1\times \PP^1$'s, six $\PP^5$'s, 18 Hirzebruch surfaces $\mathcal{F}_1$, and 36 embedded fat points, arranged into a total of three connected components.  See Proposition~\ref{prop:f4p333}.  In particular, we'll see that $\bF_4(P^3_{3,3})$ has components not of the form $\comp_4(\sigma,\tau)$, and that this Fano scheme can have embedded primary components.

\begin{table}
  \begin{center}
    \begin{tabular}{c c c c c c c c c }  
    \hline
      $n$& $2$& $3$& $4$& $5$& $6$& $7$ & $8$ & $9$ \\
\hline
\noalign{\vskip .5mm}
      Non-empty iff $k\leq$&  $1$&$5$&$11$&$19$&$29$&$41$&$55$&$71$   \\
\noalign{\vskip .5mm}
 \hline
 \noalign{\vskip .5mm}
      Singular iff $k\leq$& -- &$4$& $9$& $16$& $25$&$36$&$49$ &$64$     \\
\noalign{\vskip .5mm}
\hline
\noalign{\vskip .5mm}
      Connected if $k\leq$&   -- &$3$&$8$&$13$&$21$&$29$  &$40$& $51$       \\
      or $k=$& &&&&$24$&$34$--$35$&$46$--$48$&$57,60$--$63$\\
      \noalign{\vskip .5mm}
      \hline
 \noalign{\vskip .5mm}
      Disconnected if $k=$& $1$ &$4$,$5$& 10,11& 17-19& 26--29&37--41& $42,43,45$ &$53$--$56$,    \\
         & &&&&  & &$50$--$55$&$59,65$--$71$\\            
\noalign{\vskip .5mm}
\hline
    \end{tabular}
\vspace{.5cm}

  \end{center}\caption{Properties of the Fano scheme $\bF_k(P_{n,n}^n)$  for $n\leq 9$}\label{table:perm}
\end{table}

\subsection{Related work and organization}
Let us now mention some related research.
Fano schemes have been studied in a variety of contexts. The first modern treatment was given in \cite{altman:77a}, where it was proven that the Fano scheme of lines on a smooth cubic hypersurface of dimension at least three is smooth and connected. Sufficient criteria for smoothness and connectedness of Fano schemes of lines on hypersurfaces were given in \cite{barth:81a}, and generalized to higher dimensional linear spaces in \cite{langer:97a}. While Fano schemes for generic hypersurfaces $X$ always have dimension equal to the so-called expected dimension, it was proven in \cite{harris:98a} that this is also true for all smooth hypersurfaces of sufficiently low degree. General properties of Fano schemes of linear spaces on complete intersections have been studied in \cite{debarre:98a}. 

In the  case $r=m=n$, $D_{n,n}^n$ and $P_{n,n}^n$ are both irreducible hypersurfaces. However, the above-mentioned results do not apply to these hypersurfaces, since they are not  general, and their degrees are not sufficiently low. 

There has also been considerable work classifying the $\GL_m\times \GL_n$-equivalence classes of linear spaces of non-full rank $m\times n$ matrices, i.e.~the $\GL_m\times \GL_n$ orbits in the Fano schemes $\bF_k(\dmn)$. Such a classification exists for $k$ large relative to $m,n,$ and $r$ \cite{beasley:87a, pazzis}, and for $r\leq 4$ \cite{atkinson:83a, eisenbud:88a}. These classifications only apply to a limited range of $k,m,n,$ and $r$, however, and do not fully describe the geometry of $\bF_k(\dmn)$. See also e.g.~\cite{draisma} for constructions of certain maximal linear spaces of singular matrices.
Linear spaces of skew-symmetric matrices have also been studied; see e.g. \cite{manivel:05a}, where $2$-planes of $6\times 6$ rank four skew-symmetric matrices have been classified.
We do not know of any previous work studying $\fp$.

We use two main tools to study $\bF_k(\dmn)$ and $\bF_k(\pmn)$. First, there are natural torus actions on both of these Fano schemes; we are able to get a lot of geometric information by studying the local structure of these Fano schemes at torus fixed points. Torus actions and fixed points are discussed in Section~\ref{sec:torus}; some of our local computations were carried out explicitly using the \texttt{Macaulay2} package \cite{ilten:versal} in Section~\ref{sec:3}. Our second main tool consists of using deformation theory to calculate tangent space dimensions at select points of the Fano schemes. We carry out these calculations in Section \ref{sec:tangent}.

In Section \ref{sec:compression}, we study compression spaces in more detail. 
We put a lot of pieces together in Section \ref{sec:smooth} to 
discuss irreducibility, smoothness, and connectedness, in particular completing the proofs of Theorems \ref{thm:smooth}, \ref{thm:connected}, and \ref{thm:smoothp}.

Section \ref{sec:lines} is concerned with the Fano scheme of lines: we prove Theorem \ref{thm:lines} and compute the degrees of $\bF_1(D_{n,n}^n)$ for $n\leq 6$. In Section \ref{sec:3}, we discuss the explicit examples of $\bF_k(D_{3,3}^3)$ and $\bF_k(P_{3,3}^3)$. In particular, we show that in general, neither $\bF_k(\dmn)$ nor $\bF_k(\pmn)$ are reduced.
In Section \ref{sec:conclusion}, we conclude with a comparison between the cases of determinants and permanents and present some conjectures and further questions on  the structure of our Fano schemes.

\bigskip

{\noindent {\bf Acknowledgments.}}
We would like to thank Allen Knutson, Luke Oeding, and Bernd Sturmfels for helpful discussions, and Bernd Sturmfels for suggesting this problem to us.  We are grateful to the referee for very detailed comments.  The first author was supported by a NSF Postdoctoral fellowship.  The second author was partially supported by an AMS-Simons travel grant.

\section{Compression spaces and torus fixed points}\label{sec:torus}

\subsection{Compression spaces}

The first thing we would like to do is to define some very important subschemes of $\fd$ and $\fp$ arising from {\em compression spaces.}

\begin{defn} [cf.   \cite{eisenbud:88a}]\label{defn:compression}
  Fix a natural number $0\leq s \leq r-1$.
An {\em $s$-compression space} is the space of all $m\times n$ matrices over $\KK$ sending a fixed $(s+1+n-r)$-dimensional subspace $V$ of $\KK^{n}$ to a fixed $s$-dimensional subspace $W$ of $\KK^m$.
\end{defn}  
\noindent See~\eqref{fig:compression} for an example.  Now, every $r$-dimensional subspace of $\KK^n$ meets $V$ in dimension at least $s+1$, so a matrix in an $s$-compression space ``compresses'' the image of each $r$-dimensional subspace into an $r-1$ dimensional subspace.  In other words, every matrix in a compression space has rank less than $r$, and we may view this compression space as a point of the Fano scheme $\bF_{\kappa(s)}(\dmn)$, where $\kappa$ was defined in~\eqref{eqn:kappa}.
The set of all $s$-compression spaces forms a closed subscheme $\comp(s)$ of $\bF_{\kappa(s)}(\dmn)$ of dimension $\delta(s)$.
In fact, we have that $\comp(s)$ is isomorphic to 
$$\Gr(s+n-r+1,n)\times\Gr(s,m).$$  
Indeed, the morphism sending a pair $(A,B)$ of $(s+n-r+1)$- and $s$-dimensional planes in $\KK^n$ and $\KK^m$ to the space of matrices which map $A$ to $B$ is bijective, and an inspection of the natural affine charts of $\bF_{\kappa(s)}(D_{m,n}^r)$ shows that this map is an isomorphism. In  Theorem \ref{thm:comp}, we will even compute the degree of $\comp(s)$ as a subscheme of the Grassmannian $\Gr(\kappa(s)+1, mn)$ in its Pl\"ucker embedding.

Next, for each $k\le \kappa(s)$, the subscheme $\comp(s)$ of $\bF_{\kappa(s)}(\dmn)$ induces a subscheme $\comp_k(s)$ of $\bF_{k}(\dmn)$ whose points correspond to the $k$-planes sitting inside some $s$-compression space.  More precisely, 
let $\U(s)$ denote the restriction of the universal bundle on $\bF_{\kappa(s)}(\dmn)$ to $\comp(s)$. 
Then there is a natural morphism
$$
\rho_k(s):\Gr(k+1,\U(s))\to \bF_k(\dmn)
$$
from the Grassmann  bundle $\Gr(k+1,\U(s))$ to $\fd$
sending a $(k+1)$-dimensional subspace of $\KK^{mn}$ to the corresponding point of $\bF_k(\dmn)$. We denote the image of this map by $\comp_k(s)$.  We call $\comp_k(s)$ an $s$-compression component of $\fd$.  We will see in Theorem~\ref{thm:rho} and Corollary~\ref{cor:comp} that $\comp_k(s)$ is in fact an irreducible component of $\fd$.

\vspace{\baselineskip}
\noindent {\bf The permanental case.}
   In contrast to the determinantal case, not every $s$-compression space consists of matrices with vanishing $r\times r$ permanents. However, the ones that actually correspond to matrices with a fixed $(s+1+n-r)\times (m-s)$ zero submatrix do correspond to points of $\bF_k(\pmn)$. We call these {\em standard} compression spaces.  More precisely: 
   \begin{defn}  \label{defn:comps} Let $\sigma$ and $\tau$ be subsets of $\{1,\ldots,n\}$ and $\{1,\ldots,m\}$ such that $|\sigma|-|\tau|=n-r+1$.
     Let $\comp(\sigma,\tau)$ denote the compression space consisting of those linear maps $\KK^n\to\KK^m$ which send the standard basis vectors indexed by elements of $\sigma$ to the subspace of $\KK^m$ generated by  standard basis vectors indexed by elements of $\tau$. Such compression spaces are called \emph{standard compression spaces}.
 \end{defn}
\noindent A standard compression space is shown in (\ref{fig:compression}).  In that example, $\sigma = \{1,\ldots, s+1+n-r\}$ and $\tau = \{m-s+1,\ldots, m\}$.  A straightforward calculation contained in  Proposition~\ref{prop:fixed} shows that every point of $\comp(\sigma,\tau)$ lies in $\pmn$, and hence $\comp(\sigma,\tau)$ corresponds to a point of $\bF_{k}(\pmn)$ for $k=\kappa(|\tau|)$.

 Just as in the determinantal case, $\comp(\sigma,\tau)$ induces subschemes of Fano schemes $\fd$ for $k\leq \kappa(|\tau|)$.   
Regarding $\comp(\sigma,\tau)$ as a vector space, there is an obvious map
$\Gr(k+1, \comp(\sigma,\tau)) \cong Gr(k+1,\kappa(s)+1) \rightarrow
\Gr(k+1,mn)$
which is an isomorphism onto its image.  We call this image $\comp_k(\sigma,\tau).$  But the image of the map is clearly contained in $F_k(\pmn)$, i.e.~$\comp_k(\sigma,\tau)$ is a subscheme of $F_k(\pmn)$ isomorphic to a Grassmannian.  We will give conditions in Proposition~\ref{prop:perm} under which $\comp_k(\sigma,\tau)$ is actually a component of $\fp$.

\subsection{Torus fixed points.}
Next, we would like to take advantage of a natural torus action on our Fano schemes.  Let us view  points of $\PP^{mn-1}$ as $m\times n$ matrices over $\KK$, up to simultaneous rescaling.  The group $\GL_m\times \GL_n$ then acts naturally on $\PP^{mn-1}$: the first factor by inverse matrix multiplication the right, and the second factor by matrix multiplication on the left.  (We will assume throughout this paper that $m\ge 2$, since otherwise $\dmn$ and $\pmn$ become trivial.)

Now $\dmn$ is invariant under this action, but $\pmn$ is not. However, both are invariant under the action of the subgroup $T_m\times T_n$ of $\GL_m \times \GL_n$. 
Here, $T_m \cong (\KK^*)^m$ and $T_n\cong (\KK^*)^n$ are the standard diagonal tori of $\GL_m$ and $\GL_n$ that act by rescaling the rows (respectively the columns) of an $m\times n$ matrix.
This action of $T_m\times T_n$ extends naturally to actions on the Grassmannian $\Gr(k+1,mn)$ and its subschemes $\bF_k(\dmn)$ and $\bF_k(\pmn)$.

When is a closed point of $\Gr(k+1,mn)$ a torus fixed point? Such a point $Q$ corresponds to a $(k+1)$-dimensional linear subspace of $m\times n$ matrices.  But any linear relation on the entries that has more than one term cannot be preserved under all possible rescaling of rows and columns.  So $Q$ must consist of all matrices obtained by setting all but $k+1$ specified entries to zero.  
We often represent such a torus fixed point by a matrix with zeroes and $*$s (or zeroes and blanks), where a $*$ denotes an entry that can vary freely.  

We now prove that of the torus fixed points in $\Gr(k+1, mn)$, the ones that lie in $\bF_k(\dmn)$ and in $\bF_k(\pmn)$ are precisely the ones that are subspaces of {\em standard compression spaces} (Definition~\ref{defn:compression}).
\begin{prop}\label{prop:fixed}
  Let $Q$ be a $(T_m\times T_n)$-fixed point of $\Gr(k+1, mn)$.  Then the following are equivalent:
  \begin{enumerate}[(a)]
  \item $Q$ is contained in $\bF_k(\dmn)$.
  \item $Q$ is contained in  $\bF_k(\pmn)$.
  \item The linear space of matrices corresponding to $Q$ is a subspace of a standard compression space.
  \end{enumerate}
\end{prop}
\begin{proof}
  Represent $Q$ as a matrix of $k+1$ starred entries and $mn-k-1$ zero entries, as explained above.  Now make a bipartite graph $G$ on vertex sets $V=\{1,\ldots,m\}$ and $V'=\{1,\ldots,n\}$ with an edge from $i\in V$ to $j\in V'$ whenever entry $ij$ is starred.  Now, let us prove that a series of statements are equivalent, the first one being that $Q$ is contained in $\fd$ (or $\fp$) and the last being that $Q$ is a subspace of a standard compression space. 
  
  First, $Q$ is contained in $\fd$ (respectively $\fp$) if and only if for every subset $W$ of $V$ of size $r$, $G$ fails to have a matching of $W$ into $V'$. After all, if $G$ fails to have such a matching, then every term from each $r\times r$ determinant has a zero entry in it; whereas if $G$ has such a matching, then specializing the $r$ entries corresponding to the edges in that matching to $1$, and setting all others to zero, gives a nonzero $r\times r$ determinant (or permanent).  

  Second, by Lemma \ref{lemma:hall} below,
   $G$ fails to have a matching of such a $W\subset V$ into $V'$ if and only if there exist some $m-s$ vertices in $V$ with only $r-s-1$ neighbors amongst the vertices in $V'$.  
  But the existence of some $m-s$ vertices in $V$ with only $r-s-1$ neighbors in $V'$ means precisely that $Q$ has an $(m-s)\times (n-r+s+1)$ block of zeroes, i.e.~$Q$ is a subspace of a standard $s$-compression space.
\end{proof}
\begin{lemma}\label{lemma:hall}
Let $G=(V,V',E)$ be a bipartite graph and $r$ a number with $1\leq r \leq |V|$. Then the following are equivalent:
\begin{enumerate}
	\item There is no subset $W\subset V$ with $|W|=r$ which has a matching into $V'$.\label{item:propone}
	\item There is a subset $S\subset V$ having only $r+|S|-|V|-1$ neighbors in $V'$.
\end{enumerate}
\end{lemma}
\begin{proof}
	Add $|V|-r$ vertices $U$ to $G$, each of which is adjacent to every vertex in $V$. Then condition~(\ref{item:propone}) above is equivalent to $V$ having no matching into $U\cup V'$. By Hall's Marriage Theorem, this is equivalent to the existence of a subset $S$ of vertices in $V$ with fewer than $|S|$ neighbors in $U\cup V'$.
Now, these $S$ vertices have $|V|-r$ neighbors in $U$, hence fewer than $|S|+r-|V|$ neighbors in $V'$.
\end{proof}

\begin{rem}\label{rem:intersect}
  Proposition \ref{prop:fixed} implies that any irreducible component $Z$ of $\bF_k(\dmn)$ or $\bF_k(\pmn)$ must intersect a subscheme of the form $\comp_k(s)$ or $\comp_k(\sigma,\tau)$, respectively. Indeed,  our Fano schemes are projective, hence $Z$ is as well, and thus must contain a torus fixed point. But any torus fixed point is contained in a subscheme of the desired form.  This simple observation will play an important role in our proofs.
  \end{rem}

Finally, we can easily prove when $\fd$ and $\fp$ are nonempty just by looking at torus fixed points.  See \cite{dieudonne} for the determinantal case.  
\begin{prop}\label{prop:nonempty}
The Fano schemes $\fd$ and $\fp$ are nonempty if and only if $k<(r\!-\!1)n$.
\end{prop}
\begin{proof}
The schemes $\fd$ and $\fp$ are nonempty if and only if $k \le \kappa(s)$ for some $s\in\{0,\ldots,r-1\}$.  After all, if $k\le\kappa(s)$ for some $s$, then $\fd$ and $\fp$ each contain a torus fixed point that is a subspace of a standard $s$-compression space; but if not, then $\fd$ and $\fp$ have no torus fixed points so are empty.

Now, we have $$\kappa(0) = (r-1)m-1 \quad\textrm{ and }\quad \kappa(r-1) = (r\!-\!1)n-1.$$  
So if $k<(r\!-\!1)n$, then $k\le \kappa(r-1)$ as desired.  For the converse, if $k \ge (r\!-\!1)n$ then $k > \kappa(0)$ and $k > \kappa(r-1)$.  Then by convexity of $\kappa$, we conclude that $k> \kappa(s)$ for all $s$.
\end{proof}

\section{Tangent Space Calculations}\label{sec:tangent}

The next thing we would like to do is to calculate the dimension of the tangent space of $\fd$, respectively $\fp$, at some carefully chosen points.  The points we will look at lie in $\comp_k(s)$ or $\comp_k(\sigma,\tau)$, i.e.~they are $k$-planes inside compression spaces; and we will choose them to have a particular form that is favorable to making explicit computations using deformation theory.  

Our main theorems in this section (Theorem \ref{thm:ts} for the determinant and Theorem~\ref{thm:tsperm} for the permanent) will each have two parts.  First, the tangent space at a {\em general} point of $\comp_k(s)$ or $\comp_k(\sigma,\tau)$ can be no larger than its dimension at the points we study, since tangent space dimensions are upper semicontinuous.  Second, we will show that if certain additional inequalities hold, then every torus fixed point actually has the special form that allows our computation to go through.  We can therefore conclude that tangent space dimension at {\em every} point of $\comp_k(s)$ or $\comp_k(\sigma,\tau)$ can be no larger than the specified dimension.  

This section is the technical heart of the paper, and we will reap the rewards in Sections~\ref{sec:compression} and~\ref{sec:smooth} when we use these tangent space calculations to prove theorems on smoothness and connectedness of $\fd$ and $\fp$.

Here are the two main theorems.  

\begin{theorem}\label{thm:ts}
Fix integers $k$ and $s$ with $1\leq k \leq \kappa(s)$ and $0\leq s \leq r-1$.
\begin{enumerate}[(i)]
	\item 
  For a general point $\eta\in \comp_k(s)$, 
\begin{align*}
  \dim T_\eta\bF_k(\dmn)\leq \delta(s)+(k+1)(\kappa(s)-k).
\end{align*}
where $\kappa(s)$ and $\delta(s)$ are defined in~\eqref{eqn:kappa} and~\eqref{eqn:delta}.  

\item If furthermore both of the conditions 
\begin{align}
  k> \kappa(s)-(m-s-1)\label{eq:smootha}&& \textrm{if $s \ne r-1$}\\
  k> \kappa(s)-(n-r+s)\label{eq:smoothb}&& \textrm{if $s \ne 0$}
\end{align}
hold, then the dimension bound holds for {\em every} point $\eta\in \comp_k(s)$.  
\end{enumerate}
\end{theorem}
We have the following analogous result for the permanent:

\begin{theorem}\label{thm:tsperm}

Fix integers $k$ and $s$ such that
\begin{align*}
s=0,~r-1\qquad &\textrm{ and }\qquad 2\leq k \leq \kappa(s),\textrm{ or}\\
1\leq s\leq r-2\qquad &\textrm{ and }\qquad 5\leq k \leq \kappa(s).
\intertext{Suppose further that}
s+1+n-r \ge 3 \quad\text{  if }\quad s\ne 0\qquad &\text{ and }\qquad m-s \ge 3\quad\text{ if }\quad s\ne r-1.
\end{align*}
Consider a standard compression space $\comp(\sigma,\tau)$ (Definition~\ref{defn:comps}) with $|\tau|=s$.
\begin{enumerate}[(i)]
	\item   For a general point $\eta\in \comp_k(\sigma,\tau)$, 
\begin{align*}
  \dim T_\eta\bF_k(\pmn)\leq (k+1)(\kappa(s)-k).
\end{align*}
\item If furthermore both of the conditions 
\begin{align}
  k> \kappa(s)-(m-s-2)\label{eq:smoothap}&& \textrm{if $s \ne r-1$}\\
  k> \kappa(s)-(n-r+s-1)\label{eq:smoothbp}&& \textrm{if $s \ne 0$}
\end{align}
hold, then the dimension bound holds for {\em every} point $\eta\in \comp_k(\sigma,\tau)$. 
\end{enumerate}
\end{theorem}

\begin{rem}
  For Theorem \ref{thm:ts},  the bounds on $k$ in \eqref{eq:smootha} and \eqref{eq:smoothb} are sharp. Indeed, if $k$ doesn't satisfy both bounds, then there are torus fixed points in $\comp_k(s)$ which also lie in $\comp_k(s-1)$ or $\comp_k(s+1)$. Since these schemes are not equal by Proposition \ref{prop:distinct} and the dimension of $\comp_k(s)$ is $\delta(s)+(k+1)(\kappa(s)-k)$ by Corollary \ref{cor:comp}, the tangent space dimension at such points must be higher.
\end{rem}

We will prove our two main theorems in parallel below, because much of the setup is the same.

We start by recalling a completely general characterization of tangent spaces to Fano schemes. Suppose $X\subset\PP^N_K$ is any projective scheme with homogeneous ideal $I\subset S=K[x_0,\ldots,x_N]$, and $\Lambda\subset\PP^N$ is a $k$-plane contained in $X$ defined by the homogeneous ideal $J\subset S$. Write $[\Lambda]$ for the corresponding point of $\bF_k(X)$ and let $\J$ be the ideal sheaf of $\Lambda$.   Then the tangent space $T_{[\Lambda]} \bF_k(X)$ is isomorphic to the space of first-order deformations of $\Lambda$ in $X$, which is well-known to be isomorphic to $\Hom_{\CO_{X}}(\J,\CO_{\Lambda})$, see e.g.~\cite[Theorem 2.4]{hartshorne:10a}. Now, in our case, $J$ is generated by linear forms, so $J/I$ as well as $S/J$ are saturated graded $S/I$-modules.  So $T_{[\Lambda]} \bF_k(X)\cong \Hom _{\CO_{X}}(\J,\CO_{\Lambda}) \cong \Hom_{S/I}(J/I,S/J)_0$, the space of degree-preserving maps of $S/I$-modules.

Now, in our specific situation, fix $r,m,$ and $n$, and let $S=\KK[x_{1,1},\ldots,x_{m,n}]$.  
Suppose that $J \subset S$ is a linear ideal defining a $k$-plane $\Lambda\subset\PP^{mn-1}$ of matrices lying in the standard $s$-compression space shown in \eqref{fig:compression}.  
Now, let us choose standard linear monomials $z_0,\ldots,z_k$ for $S/J$, i.e.~specify an isomorphism $S/J\cong \KK[z_0,\ldots,z_k]$.  Then we can regard $\Lambda$ as a matrix whose entries are linear forms in $z_0,\ldots,z_k$, such that the linear span of these forms has full dimension $k+1$.  
In other words, $\Lambda$ has the form
\begin{equation}
\label{eq:matrix}
\Lambda=\left(\!\begin{array}{c |c}
     &\cr
	\quad \text{\Large 0} \quad &B\cr
	&\cr
	\hline
	C&D
\end{array}\!\right)
\end{equation}
where the upper left block of zeroes has size $(m-s)\times (s+1+n-r)$, and the entries of $B,C,$ and $D$ are in $\KK[z_0,\ldots,z_k]_1$.  
See Example~\ref{ex:detform} for a matrix of this form.

The next lemma tells us concretely how to compute the dimensions of the tangent spaces of $\fd$ and $\fp$ at the point $[\Lambda]$.  Call an $r\times r$ submatrix of an $m\times n$ matrix {\em anchored} if it involves the last $s$ rows as well as the last $r-s-1$ columns.

\begin{lemma}\label{lemma:abc}
For a $k$-plane $\Lambda$ as above,
let $a_{\det}$ (respectively $a_{\perm}$) be the $\KK$-dimension of the space of $(m\!-\!s)\times(s\!+\!1\!+\!n\!-\!r)$ matrices $A$ with entries in $\KK[z_0,\ldots,z_k]_1$ such that the $r\times r$ anchored determinants (respectively permanents) of the $m\times n$ matrix
\begin{equation}\label{eq:abc}
Q=\left(\!\begin{array}{c |c}
     &\cr
	\quad A \quad &B\cr
	&\cr
	\hline
	C&0
\end{array}\!\right)
\end{equation}
all vanish.
Then 
\begin{align*}
\dim T_{[\Lambda]} \fd &= a_{\det} + (k+1)(\kappa(s)-k),\\
\dim T_{[\Lambda]} \fp &= a_{\perm} + (k+1)(\kappa(s)-k).
\end{align*}
\end{lemma}

\begin{proof}[Proof of Lemma~\ref{lemma:abc}]
Let $I_{\det}\subset S$ and $I_{\perm}\subset S$ be the ideals generated by the determinants and permanents, respectively, of the $r\times r$ submatrices of the $m\times n$ matrix $(x_{i,j})$.  We will treat both cases simultaneously.  Let $I=I_{\det}$ or $I_{\perm}$, and let $J\subset S$ be the linear ideal of $\Lambda.$  Then $J$
  is generated by the coordinates $x_{i,j}$ with $i\le m-s$ and $j\le s+1+n-r$, corresponding to the condition that the upper left block is zero, along with $\kappa(s)-k$ additional independent linear forms $f_i$ on the remaining $x_{i,j}$'s that determine the linear relations on the other entries of the matrix. 
  We wish to compute the dimension of $\Hom_{S/I}(J/I,S/J)_0$.

   To give an element of $\Hom_{S/I}(J/I,S/J)_0$, we must specify an element of $(S/J)_1$, i.e.~a linear form in $z_0,\ldots,z_k$, for each generator $x_{i,j}$ and $f_i$.  This choice of elements of $(S/J)_1$ is subject to the condition that all the relations among the $x_{i,j}$'s and the $f_i$'s are preserved among their images.
Let's examine these relations.

Assume for a moment that $r=m=n$, and suppose we have a map $\phi\in \Hom_{S/I}(J/I,S/J)_0$.  
  Now, among the relations on generators of $J/I$, there are the Koszul relations of the form $a_1(a_2) - a_2(a_1)$ for $a_i \in J$, which $\phi$ obviously send to $0 \mod J$.  Apart from these, then, the only nontrivial relation among generators of $J/I$ comes from writing the determinant or permanent of the $r\times r$ matrix $(x_{i,j})$ as an $S$-linear combination of the generators of $J$.  In fact, it is easy to choose such an $S$-linear combination (uniquely up to the Koszul relations), since every term of the $r\times r$ determinant (or permanent) contains at least one position in the $\bf{0}$ block and hence one generator of $J$.  
  Then the fact that $\phi$ sends the $r\times r$ determinant or permanent to zero means that
  \begin{equation}\label{eqn:rel}
  \sum_{i,j}\pm (b_i c_j)\phi(x_{i,j}) = 0,
\end{equation}
where the sum above has one term for each entry $(i,j)$ in the upper left zero block.  Here, $b_i$ is the determinant or permanent of the submatrix of $B$ gotten by deleting the $i$th row, and $c_j$ is the determinant or permanent of the submatrix of $C$ gotten by deleting the $j$th column.  The reason that no other terms of the determinant or permanent appear is that every other term contains at least two positions in the $\mathbf{0}$ block, and hence, after factoring out some generator of $J$, will still vanish modulo~$J$.  So \eqref{eqn:rel} is the only constraint on the images of the generators of~$J$.

Now, if instead $r<n$, we simply note that 
the nontrivial relations among the images of the $x_{i,j}$'s are just the relation \eqref{eqn:rel} obtained by replacing~\eqref{eq:matrix} by each of its $r\times r$ anchored submatrices.  (The non-anchored submatrices do not add any constraints, because each term of a non-anchored $r\times r$ determinant or permanent contains at least two positions in the ${\bf 0}$ block and so is necessarily sent to $0$ in $S/J$.)  
Then condition \eqref{eqn:rel} applied to the $r\times r$ submatrices of $\Lambda$ can be rewritten as the condition that the determinants or permanents of the $r\times r$ anchored submatrices of the $m\times n$ matrix in \eqref{eq:abc} vanish, where $A=(\phi(x_{i,j}))$.

Now we can easily count dimensions for $\Hom_{S/I}(J/I,S/J)_0$.  We have just shown that the dimension of the space of choices of $\phi(x_{i,j})$, for $i\le m\!-\!s$ and $j\le s\!+\!1\!+\!n\!-\!r$, is precisely $a_{\det}$ (respectively $a_{\perm}$).  Next, we have exactly $k+1$ dimensions for choosing each $\phi(f_i)$, since after all the $f_i$ do not appear at all in \eqref{eqn:rel}, and $\dim(S/J)_1=k+1$.
So   the dimension of space of choices for the $\phi(f_i)$ is $(k+1)(\kappa(s)-k)$.  Adding, we have proved Lemma~\ref{lemma:abc}.
\end{proof}

\noindent We now prove part (i) of each theorem by applying Lemma~\ref{lemma:abc} to some carefully chosen points~$[\Lambda]$.
\begin{proof}[Proof of Theorem~\ref{thm:ts}(i)]  To prove Theorem~\ref{thm:ts}(i), we need to exhibit, for any $k$ and $s$ as in the statement of Theorem~\ref{thm:ts},  a matrix of the form 
\eqref{eq:matrix}
such that the dimension $a_{\det}$ defined in Lemma~\ref{lemma:abc} is exactly $\delta(s)$.  Then \ref{thm:ts}(i) would follow by upper semicontinuity.

First, some notation: for distinct variables $z_i$ and $z_j$, define matrices
\begin{equation}\label{eq:bc}
B(z_i,z_j)=\left(\begin{array}{c c c} 
  z_i&&\\
z_j&\ddots&\\
&\ddots&z_i\\
&&z_j\\
\hline
&&\\
&\text{\Large 0}&\\
&&
\end{array}
\right)
\quad \textrm{and}
\quad C(z_i,z_j)=\left(\begin{array}{c c  c c|}
  z_i&z_j&&\\
&\ddots&\ddots&\\
&&z_i&z_j\\
\end{array}
\ \text{\Large 0}\ 
\right)
\end{equation}
of dimensions $(m-s)\times(r-s-1)$ and $s\times (s+1+n-r)$ respectively.  (Thus, by slight abuse of notation, the dimensions will depend on $s$.)
Now let's specify the form that $\Lambda$ should take, in four cases.  

If $s=0$, so the matrices $C$ and $D$ don't appear, we let 
$B=B(z_0,z_1)+B',$ where $B'$ is any matrix whose entries are linear forms not involving $z_0$ or $z_1$, and such that the linear span of the forms appearing in $B$ has full dimension $k+1$.

Similarly, if $s=r-1$, so that the matrices $B$ and $D$ don't appear, then we let $C=C(z_0,z_1)+C'$,
where $C'$ is any matrix whose entries are linear forms not involving $z_0$ or $z_1$, and such that the linear span of the forms appearing in $C$ has full dimension $k+1$.

Next, if $s\ne 0, r-1,$ and $k=1$, we let $B=B(z_0,z_1)$ and $C = C(z_0, z_1)$.  There are no restrictions on $D$.

Finally, if $s\ne 0, r-1,$ and $k\ge2$, we let $B=B(z_0,z_1)+B'$ and $C=C(z_1,z_2)+C'$, for any matrices $B'$ and $C'$ that do not involve $z_0,~z_1,$ or $z_2$, and furthermore do not involve any variables in common.  There are again no restrictions on $D$, except that $B'$, $C'$, and $D$ must be chosen so that the linear span of the forms appearing in $\Lambda$ has full dimension $k+1$.

\begin{ex}\label{ex:detform}
For $k=4$, $s=2$, $r=m=5$, and $n=6$, the following $5\times 6$ matrix over $\KK=\mathbb{C}$, say, has the desired form.  The $*$ entries represent arbitrary linear forms in $z_0,z_1,z_2,z_3,$ and $z_4$.
$$\left(
\begin{array}{c c c c c c}
0 &0& 0& 0 &z_0 &z_3\\
0 &0& 0& 0 &z_1\!-\!z_3 &z_0\!+\!3z_3\\
0 &0& 0& 0 &5z_3 &z_1\!+\!11z_3\\
z_1& z_2 &0& 0 &* &*\\
0& z_1&z_2& z_4 &* &*
\end{array}
\right)
$$
\end{ex}
\noindent Now in each case, we are done if we can show that the space of matrices $A$ satisfying that the $r\times r$ anchored minors of the matrix $Q$ in \eqref{eq:abc} vanish has dimension exactly $\delta(s)$.  In fact, we can identify a space of dimension $\delta(s)$: take $A$ to be any sum of a matrix in the $\KK$-column span of $B$ and a matrix in the $\KK$-row span of $C$.  
The resulting matrix $Q$ still has rank $<r$. Furthermore, $B$ and $C$ have full rank, but on the other hand no nonzero matrix is both in the $\KK$-column span of $B$ and the $\KK$-row span of $C$.  So the space of matrices $A$ obtained in this way forms a $\KK$-vector space of dimension exactly $(r-s-1)(s+1+n-r)+s(m-s)=\delta(s)$.  We have thus reduced to proving the converse: if the anchored $r\times r$ minors of $Q$ vanish, then $A$ is the sum of a matrix in the $\KK$-column span of $B$ and a matrix in the $\KK$-row span of~$C$.

The degenerate cases $s=0,~r-1$ are thus an immediate consequence of the lemma below, whose proof we postpone to the end of the section.

\begin{lemma}\label{lemma:rowdet}
Let $2\leq m\leq n$.
Consider an $m\times n$ matrix 
$$P=\left(\begin{array}{c c c}
  &\vec a&\\
  \hline
  &&\\
  &C&\\
  &&
  \end{array}
\right)$$
whose entries are linear forms $z_0,\ldots,z_k$, where $\vec a=(a_i)$ is a $1\times n$ matrix and $C$ is an $(m-1)\times n$ matrix.
Suppose $C=C(z_0,z_1)+C'$, where $C(z_0,z_1)$ is an $(m-1)\times n$ matrix as shown in \eqref{eq:bc}, and $C'$ only involves $z_2,\ldots, z_k$.  

If the maximal minors of $P$ vanish, then $\vec a$ must lie in the $\KK$-linear rowspan of~$C$.
\end{lemma}

\noindent Thus, we may now assume $1\leq s \leq r-2$.  We have two remaining cases: $k=1$ and $k\ge 2$.  Let us first assume $k\ge 2$.  

By subtracting columns of $B$ from $A$, we may assume that in the first $r-s-1$ rows of $A$, no $z_0$ appears.  Also, let us add a copy of the $(r-s)$ row to each row $r-s+1,\ldots, m-s$.  Then we may assume that each row $r-s+1,\ldots, m-s$ in the matrix $B$ has no $z_0,z_1$'s appearing except for a single $z_1$ in the last column.  We now claim that each row of $A$ is in the $\KK$-linear span of the rows of $C$. 

First, let $Q'$ be any $r\times n$ submatrix of $Q$ obtained by choosing the first $r-s-1$ rows, one middle row, and the last $s$ rows of $Q$.  Now consider any maximal submatrix $\tilde Q$ of $Q'$ that involves the last $r-s-1$ columns of $Q'$.  Write
\begin{equation}\label{eqn:Q'}
  \tilde Q=\left(\!\begin{array}{c c}
	\tilde A &\tilde B\\
	\tilde C&\mathbf{0}
\end{array}\!\right),
\end{equation}
where $\tilde A$, $\tilde B$, and $\tilde C$ are submatrices of $A,B,$ and $C$, respectively.  
For $i=1,\ldots,r-s$, let $b_i$ be the determinant of the matrix gotten from $\tilde B$ by deleting its $i$th row, and let $\tilde c_i$ be the determinant of the matrix gotten by stacking the $i$th row of $\tilde A$ to $\tilde C$. Let us show that each $\tilde c_i=0$.

Since $Q$ had rank at most $r-1$ by assumption, we have $\det(\tilde Q)=0.$  Note that
$$0 = \det(\tilde Q)=b_1\tilde c_1 - b_2\tilde c_2 +  \cdots \pm b_{r-s}{\tilde c}_{r-s}.$$
Now each $b_i$ is of the form $z_0^{i-1}z_1^{r-s-i}$ plus terms that have lower total degree in $z_0$ and $z_1$.  Also, each ${\tilde c}_i$, except possibly ${\tilde c}_{r-s}$, has no $z_0$ term appearing, by assumption.   
Then by inspecting the degree in $z_0$, we conclude that $\tilde c_{r-s}=0$.  

Next, we claim that no other ${\tilde c}_i$ has $z_1$ appearing: for if $z_1$ appears with highest degree $d$ in ${\tilde c}_i$, 
 then the factor $z_0^{i-1}z_1^{r-s-i+d}$ would appear in the summand $b_i {\tilde c}_i$ but nowhere else.  So $z_0$ and $z_1$ do not appear in any ${\tilde c}_i$.  Finally, by inspecting the $(z_0,z_1)$-degree of each summand, we conclude that each ${\tilde c}_i=0.$  

Now, by ranging over all possible choices of $Q'$ and $\tilde Q$, we have that 
$$
\rank \left(
\begin{array}{c}
A\\
C
\end{array}\right)=\rank C.
  $$
  Then we may apply 
Lemma \ref{lemma:rowdet} to conclude that
 the rows of $A'$ are actually $\KK$-linear combinations of the rows of $C$. 

Finally, assume $1\le s\le r-2$ and $k=1$.  Once again, we want to show that any matrix $A$ satisfying that the $r\times r$ anchored minors of \eqref{eq:abc} vanish must be a sum of a matrix in the $\KK$-column span of $B=B(z_0,z_1)$ and a matrix in the $\KK$-row span of~$C=C(z_0,z_1)$.

First, note that rows $r-s+1,\ldots, m-s$ of $B$ are zero.  So we will start by showing that for any $i$ with  $r-s+1 \le i \le m-s$, the $i^{th}$ row of $A$ is in the $\KK$-rowspan of $C$.  Indeed,  
let $Q'$ be the $r\times n$ submatrix of $Q$ obtained by choosing the first $r-s-1$ rows, the $i^{th}$ row, and the last $s$ rows of $Q$.  So the lower right $(s+1)\times(r-s-1)$ block of $Q'$ is zero.  Now the first $r-s-1$ rows of $Q'$ are $\KK(z_0,z_1)$-linearly independent and their span intersects the span of the last $s+1$ rows trivially.  Thus the last $s+1$ rows have rank at most $s$. Applying Lemma~\ref{lemma:rowdet} to these last $s+1$ rows, we are done.  Furthermore, the same argument shows that columns $s+2,\ldots,s+1+n-r$ of $A$ are $\KK$-combinations of the columns of $B$.

Now after deleting rows $r-s+1,\ldots,m-s$ and columns $s+2,\ldots,s+1+n-r$, we are left with a singular $r\times r$ matrix $Q'$.  Let $A'$ denote the 
upper left $(r-s)\times(s+1)$-submatrix of $Q'$.  After adding multiples of the columns of $B$ and the rows of $C$, we may assume that $z_0$ appears only in the lower right entry of $A'$.  But $z_0$ cannot appear there either, for otherwise the monomial $z_0^r$ would appear in $\det(Q')$ with nonzero coefficient, contradicting that $\det(Q') = 0.$  So $A'$ has only $z_1$'s.  

Next, we claim we may reduce to the case
that $z_1$'s only appear in the first row and the last column of $A'$.  This is because if $a_{i,j} = \lambda z_1$, say, is a nonzero entry elsewhere in $A'$, then we can alternately add $\lambda$-multiples of rows of $C$ and subtract $\lambda$-multiples of columns of $B$ to move $\lambda z_1$ to any other entry in the same antidiagonal.

But in this case, $z_1$'s don't appear at all, since they would each contribute a unique nonzero monomial in $z_0^i,z_1^{r-i}$ to $\det(Q')$, but $Q'$ is singular.  Thus, after adding rows of $C$ and columns of $B$, we have arrived at an equality $A'=0$, so we are done.  This finishes the proof of Theorem~\ref{thm:ts}(i).
\end{proof}

\begin{proof}[Proof of Theorem~\ref{thm:tsperm}(i)]  Just as in the determinantal case, we need to exhibit, for any $k$ and $s$ as in Theorem~\ref{thm:tsperm},  a matrix of the form 
\eqref{eq:matrix}
such that $a_{\perm}=0$ in Lemma~\ref{lemma:abc}.  Then \ref{thm:tsperm}(i) would follow by upper semicontinuity: we need only invoke the natural $(S_m\times S_n)$-action on $\fp$ gotten by permuting rows and columns that sends any standard compression component $\comp_k(\sigma,\tau)$ to the particular one shown in (\ref{fig:compression}).

First, for distinct variables $z_i,z_j,$ and $z_k$, define matrices $B(z_i,z_j,z_k)$ and $C(z_i,z_j,z_k)$ respectively as
\begin{equation}\label{eq:bcperm}
\left(\begin{array}{c c c c} 
  z_i&&&z_k\\
z_j&\ddots&\\
z_k&\ddots&z_i\\
&\ddots&z_j&z_i\\
&&z_k&z_j\\
\hline
&&&\\
\multicolumn{4}{c}{\text{\Large 0}}\\
&&&
\end{array}
\right)
\quad \textrm{and}\quad 
\left(\begin{array}{c c c  c c|c c c}
  z_i&z_j&z_k&& &&&\\
&\ddots&\ddots&\ddots&&&\text{\Large 0}&\\
&&z_i&z_j&z_k&&&\\
z_k&&&z_i&z_j&&&\\
\end{array}
\right)
\end{equation}
of dimensions $(m-s)\times(r-s-1)$ and $s\times (s+1+n-r)$ respectively.  

Let us clarify the degenerate cases of this definition: if $s=1$, then $C(z_i,z_j,z_k)$ = $(z_i~z_j~z_k~0~\cdots~0)$, and similarly if $s=r-2$, then $B(z_i,z_j,z_k)$ = $(z_i~z_j~z_k~0~\cdots~0)^T$.  If $s=2$ then $$C(z_i,z_j,z_k)=\left(\begin{array}{cccc}
z_i&z_j&z_k&0~\cdots~0\\
z_k&z_i&z_j&0~\cdots~0
\end{array}\right);$$ and similarly for $B(z_i,z_j,z_k)$ when $r-s-1=2$.

  In particular, in each case, $C(z_i,z_j,z_k)$ has to have at least $3$ columns whenever it is defined: in other words, 
if $s\ne 0$, we will always assume that $s+1+n-r\ge 3$.  Similarly, $B(z_i,z_j,z_k)$ has to have at least $3$ rows whenever it is defined: in other words, 
if $s\ne r-1$, we will always assume that $m-s\ge 3$.  This accounts for the conditions in the statement of Theorem~\ref{thm:tsperm}.
Now we treat each case in turn.

First,
suppose $s=0$, so $m\ge 3$ and the matrices $C$ and $D$ don't appear.  
Then let $B=B(z_0,z_1,z_2)+B',$ where $B'$ is any matrix whose entries are linear forms not involving $z_0,z_1,$ or $z_2$.  
Similarly, 
if $s=r-1$, so $n\ge 3$ and $B$ and $D$ don't appear,  
then we let $C=C(z_0,z_1,z_2)+C'$,
where $C'$ is any matrix whose entries are linear forms not involving $z_0,z_1,$ or $z_2$.
In both of these cases, Theorem~\ref{thm:tsperm}(i) is an immediate consequence of the following lemma, whose proof we postpone to the end of the section.

\begin{lemma}\label{lemma:rowperm}
Let $2\leq m\leq n$ but not both $m=n=2$.
Consider an $m\times n$ matrix 
$$P=\left(\begin{array}{c c c}
  &\vec a&\\
  \hline
  &&\\
  &C&\\
  &&
  \end{array}
\right)$$
whose entries are linear forms $z_0,\ldots,z_k$, where $\vec a=(a_i)$ is a $1\times n$ matrix and $C=C(z_0,z_1,z_2)+C'$ is an $(m-1)\times n$ matrix,
where $C(z_0,z_1,z_2)$ is as shown in \eqref{eq:bcperm}, and $C'$ does not involve $z_0,z_1,$ or $z_2$.  

If the maximal permanents of $P$ vanish, then $\vec a = 0$.
\end{lemma}

\noindent So we are done with 
the cases $s=0$ and $s=r-1$,
modulo the lemma.

So we may assume for the rest of the proof that $1\le s \le r-2$.  
We let $B=B(z_0,z_1,z_2)+B'$ and $C=C(z_3,z_4,z_5)+C'$, for any matrices $B'$ and $C'$ that do not involve $z_0,\ldots,z_5$, and furthermore do not involve any variables in common.  
Let $D$ be any $s\times(r-s-1)$ matrix of linear forms.  Again, we want to show $A=0$ if the $r\times r$ anchored permanents of $Q$ in~\eqref{eq:abc} are zero.

See \eqref{eq:gencase} for a picture of the matrix $Q$. In this picture, $A$ has been broken into four submatrices. 
If $B$ or $C$ have only one or two columns (respectively rows), then they should be interpreted as having the degenerate forms described above. 
(We have also suppressed $B'$ and $C'$ in this picture for clarity.)

\begin{equation}\label{eq:gencase}
\begin{blockarray}{ccccccccc}
     & &s\!+\!1&&&n\!-\!r&&r\!-\!s\!-\!1&       \\
\begin{block}{c(cccc|c|ccc)}
                 &&&&&&z_0&&z_2\\
              r\!-\!s&&&&&&z_1&z_0& \\
              &&&&&&z_2&z_1&z_0\\
              &&&&&&&z_2&z_1\\ \cline{2-9}
              m\!-\!r&&&&&&&&\\ \cline{2-9}
              &z_3&z_4&z_5&&&&&\\
              s&&z_3&z_4&z_5&&&\text{\Large 0}&\\
              &z_5&&z_3&z_4&&&&\\
\end{block}
\end{blockarray}
\end{equation}

\noindent We break into the following two cases:
\begin{itemize}
	\item $s=1$ or $s=r-1$, and
	\item $1<s<r-2$.
\end{itemize}

First, suppose $s=1$.  So $C$ is a row matrix.
For this case, we need the following lemma, which we will prove at the end of the section.

\begin{lemma}\label{lemma:2case}
Let $r \ge 3$, and suppose $Q'$ is an $m\times r$ matrix with entries in $\KK[z_0,\ldots,z_k]_1$ of the form
\begin{equation*}
Q'=\left(\!\begin{array}{c |c}
     &\cr
	A'  &\quad B\quad\cr
	&\cr
	\hline
	\ell_1~ \ell_2&0
\end{array}\!\right).
\end{equation*}
Here $\ell_1$ and $\ell_2$ are nonzero linear forms in $z_3,\ldots,z_k$ that are not scalar multiples of each other, and $B=B(z_0,z_1,z_2) + B'$ where $B(z_0,z_1,z_2)$ is as in~\eqref{eq:bcperm} and 
$B'$ does not contain $z_0,z_1$, or $z_2$.  

If the $r\times r$ permanents of $Q'$ vanish, then for each $i=1,\ldots,m-1$, the $i^{th}$ row of $A'$ is of the form $(a_i\ell_1~-\!a_i\ell_2)$ for some $a_i\in \KK$.
\end{lemma}

\noindent Assuming the lemma holds, let us prove that $A=0$.  By applying Lemma~\ref{lemma:2case} to each $m\times r$ submatrix of~$Q$ gotten by choosing any two of the first three columns and the last $r-2$ columns, we see that each pair of the first three columns has entries that satisfy the sign relation as in Lemma~\ref{lemma:2case}.  But that is impossible unless each entry in the first three columns of $A$ is zero.  

Let $Q_{m,1}$ denote the matrix gotten by deleting the $m^{th}$ row and first column of $Q$.  Now, since the first column, say, of~$Q$ is zero except for the last entry which is nonzero, and $Q$ has vanishing $r\times r$ anchored permanents, if follows that the $(r-1)\times (r-1)$ anchored permanents of the matrix $Q_{m,1}$ vanish.  Now applying Lemma~\ref{lemma:rowperm} to (the transpose of) each $(m-1)\times (r-1)$ submatrix of $Q_{m,1}$ that involves the last $r-2$ columns shows that each column of $A$ is zero.

The case 
$s=r-2$
is exactly analogous.  In this case, the roles of $B$ and $C$ are simply reversed: now $B$ is a column matrix with at least three nonzero entries. 
Then we apply Lemma~\ref{lemma:2case} again to conclude that the first three rows of $A$ must be zero, and then apply Lemma~\ref{lemma:rowperm} to each $(m-1)\times (r-1)$ submatrix of $Q_{1,n}$ to show that every row of $A$ is zero.

For the final case, let $1<s<r-2$.  In this case, our assumptions imply that $B$ (respectively $C$) has at least two columns (respectively rows), see~\eqref{eq:gencase}.   We will call the $m-r$ rows and $n-r$ columns in~\eqref{eq:gencase}, if they are present, the {\em middle} rows and columns.  In the diagram, $A$ has been broken into four blocks, and we will show one by one that they are $0$. 

First suppose $r=m=n$, i.e.~the middle rows and columns don't appear.
As in the determinantal case, 
for $i=1,\ldots,m-s$, let $b_i$ be the permanent of the matrix gotten by deleting the $i^{th}$ row of $B$, and let  $\tilde c_i$ be the permanent of the matrix gotten by stacking the $i^{th}$ row of $A$ to $C$.  Then $$0 = \perm(Q)=b_1\tilde c_1 + b_2\tilde c_2 +  \cdots + b_{m-s}{\tilde c}_{m-s}.$$
Considering the terms of this equation of the form $$z_0^{i-1}z_1^{r-s-i}\cdot\textrm{(any monomial not involving the variables appearing in }B)$$ shows that,
after zeroing out the variables in $A$ that appear in $B$,
$\tilde c_i=0$. 
Then by Lemma \ref{lemma:rowperm}, $A$ only contains variables from $B$. 

But a symmetric argument shows that $A$ only contains variables from $C$.  Since no variables appear in both $B$ and $C$ by assumption, we must have $A=0$.

Now let us drop the assumption that $r=m=n$.  Let $Q'$ be the $r\times r$ submatrix of $Q$ gotten by removing the middle $m-r$ rows and $n-r$ columns.  Then the $r=m=n$ argument applied to $Q'$ immediately shows that the upper left $(r-s)\times(s+1)$ block of $A$ is zero.  Next, we claim that every entry in the first $r-s$ rows of $A$ is zero. Indeed, 
if in $Q'$ we replace the $(s+1)^{st}$ column with any middle column, we get an $r\times r$ matrix whose upper left $(r-s)\times s$ block is zero and whose lower left $s \times s$ block has nonzero permanent.  Then the permanent of the upper right $(r-s)\times (r-s)$ block is zero, so by Lemma~\ref{lemma:rowperm} again, the chosen middle column of $A$ is zero, proving the claim.  A symmetric argument shows that every entry in the first $s+1$ columns of $A$ is zero.

Finally, for any middle row $i$ and middle column $j$, we want to show that $a_{ij}=0$. But this is immediate from considering the matrix obtained from $Q'$ by replacing the $r-s$ row by $i$ and the $s+1$ column by $j$, and noting that the permanent of this matrix is zero yet is a nonzero multiple of $a_{ij}$.
This finishes the proof of Theorem~\ref{thm:tsperm}(i). 
\end{proof}

\begin{proof}[Proof of Theorems~\ref{thm:ts}(ii) and~\ref{thm:tsperm}(ii)]  Let us now prove part (ii) of each Theorem by showing that if $k$ is sufficiently large, the proofs of Theorems~\ref{thm:ts}(i) and~\ref{thm:tsperm}(i) can be applied to any torus fixed point of $\comp_k(s)$ or $\comp_k(\sigma,\tau)$.  When we say $k$ is sufficiently large, we mean that conditions \eqref{eq:smootha} and \eqref{eq:smoothb} are fulfilled in the determinantal case, or that conditions \eqref{eq:smoothap} and \eqref{eq:smoothbp} are fulfilled in the permanental case.

Part (ii) of each of Theorems~\ref{thm:ts} and~\ref{thm:tsperm} would then follow, because if the conditions \eqref{eq:smootha} and \eqref{eq:smoothb} or the conditions \eqref{eq:smoothap} and \eqref{eq:smoothbp} hold, then 
the bounds in \ref{thm:ts}(i) and~\ref{thm:tsperm}(i) would apply to any torus fixed point of $\comp_k(s)$, respectively $\comp_k(\sigma,\tau)$.
But the Zariski closure of the torus orbit of any point of $\comp_k(s)$ or $\comp_k(\sigma,\tau)$ certainly contains a torus fixed point.  So the bounds would hold for arbitrary points of $\comp_k(s)$, respectively $\comp_k(\sigma,\tau)$.

  Now, by the discussion in Section \ref{sec:torus}, we know that any torus fixed point $P$ of $\comp_k(s)$ or $\comp_k(\sigma,\tau)$ comes from zeroing out entries in the matrix of forms representing a standard compression space. Thus, after permuting rows and columns, we may assume that $P$ has the form \eqref{eq:matrix},
   and the entries of $B$, $C$, and $D$ are either zeroes or $*$s representing distinct variables $z_i$. Furthermore, the number of zero entries among $B,C,$ and $D$ is exactly $\kappa(s)-k$.
  
   The proofs in both the determinant and permanent cases reduce to the following combinatorics:
\begin{lemma}\label{lemma:stars}
Let $c \le p$ and $p\ge q$, and suppose we have an $p\times q$ array $B=(b_{ij})$ filled with at most $p-c$ zeroes and the rest $*$s.  Then we may permute rows and columns so that every entry $b_{i,j}$ with $i-j<c$ is a $*$.
\end{lemma}

\noindent Assuming the lemma, let us prove Theorem~\ref{thm:ts}(ii).  Suppose $s\ne r-1$, i.e.~$B$ is nonempty.  Then  $\kappa(s)-k < m-s-1$ by assumption, so $B$ is an $(m-s)\times (r-s-1)$ matrix with at most $m-s-2$ zeroes.  Then by Lemma~\ref{lemma:stars} applied with $c=2$, after permuting rows and columns of $B=(b_{i,j})$, the entries $b_{i,i}$ and $b_{i+1,i}$ are non-zero for all $i$.  
  Then a linear change of coordinates brings $B$ in the form required for Theorem \ref{thm:ts}(i). 
(The lemma shows that all entries above the diagonal $b_{i,i}$ are nonzero as well, but we don't need this.)

 If instead $s\ne 0$, so that the matrix $C$ is nonempty, and $k>\kappa(s)-(s+n-r)$, the same argument shows that there is a permutation of the rows and columns of $C$ so that each of the entries $c_{i,i}$, $c_{i,i+1}$ are non-zero for $i=1,\ldots,s$. 
   Again, a linear change of coordinates then brings $C$ into the form required for Theorem \ref{thm:ts}.
   
The proof of Theorem~\ref{thm:tsperm}(ii) is almost identical, except that we apply Lemma~\ref{lemma:stars} with $c=3$ instead of $c=2$, since the matrices in our special permanental form were tridiagonal instead of bidiagonal.  Namely: if $s\ne r-1$, so that the matrix $B$ appears, the assumption that $\kappa(s)-k < m-s-2$ says that $B$ has at most $m-s-3$ zeroes.  Then Lemma~\ref{lemma:stars} shows that after permuting, every entry $b_{i,j}$ with $i-j< 3$ is nonzero. Similarly, if $s\ne 0$, so that the matrix $C$ appears, the assumption that $\kappa(s)-k < n-r+s-1$ and Lemma~\ref{lemma:stars} again imply that after permuting, every entry $c_{i,j}$ with $j-i<3$ is nonzero.  Then we are again done after a linear change of coordinates.
So we have proved part (ii) of each theorem, and we are done, apart from the proofs of the supporting lemmas below.
  \end{proof}

\begin{proof}[Proof of Lemma~\ref{lemma:rowdet}]
 Let us write $a_i=\sum \lambda_{i,j} z_j$ for $\lambda\in\KK$.
Consider the determinant of the submatrix of $P$ consisting of its first $m$ columns. Inspection of the coefficient of the monomial
$z_0^{i}z_1^{m-i}$ for $1\leq i \leq m-1$ gives the equation
$$
\lambda_{i,0}=\lambda_{i+1,1}.
$$
Likewise, inspection of the monomials $z_0^m$ and $z_1^m$ give equations $\lambda_{1,1}=\lambda_{m,0}=0$.
Hence, after applying $\KK$-linear row operations to the first row of $P$, we may assume that $\lambda_{i,j}=0$ for $i\leq m$ and $j=0,1$. 

We now show that under this assumption, $\vec a=0$. Indeed, consider again the determinant of the submatrix of $P$ consisting of its first $m$ rows. Inspection of the coefficient of the monomial $z_0^{i-1}z_1^{m-i}z_j$ shows that $\lambda_{i,j}=0$ for $i \leq m$ and $j\ge 2$. So we have proved $a_1=\cdots=a_m=0$.  Finally, for any $i>m$, consider the determinant of the submatrix of $P$ consisting of columns $2,\ldots,m,i$. Expanding by the first row, this determinant is $\pm \, a_i\cdot(z_1^{m-1} + \textrm{ terms of lower degree in }z_1)$.  The product is zero and the second factor is nonzero, so $a_i=0.$  
\end{proof}

Before proving Lemma~\ref{lemma:rowperm}, we need the following calculation of the coefficients of the maximal permanents of a tridiagonal matrix.

\begin{lemma}\label{lemma:permterms}
Let $p\geq 2$ and  consider the $p\times(p+1)$ matrix 
$$
Q=\left(\begin{array}{c c c c c}
  u&v&w&&\\
&\ddots&\ddots&\ddots&\\
&&u&v&w\\
w&&&u&v\\
\end{array}
\right)
$$
for independent linear forms $u,v,w$, with $0$ at every other position. Denote the permanent of the submatrix obtained by deletion of the $i$th column by $q_i$. For $1\leq j\leq p-1$, the coefficient of $u^jv^{p-(j+1)}w$ in $q_i$ is equal to 
$$
\coeff_{u^jv^{p-(j+1)}w}(q_i)=\begin{cases}
  \overline{p-j} & i=j\\
  0 & i\neq j.
\end{cases}
$$
\end{lemma}
\noindent Here, $\overline{p-j}$ just means the integer $p-j$ reduced modulo the characteristic of $\KK$.

\begin{proof}
  Fix $i$, and let $y$ be any monomial appearing in $q_i$ of the form $u^jv^{p-(j+1)}w$ for some $1\leq j \leq p-1$. We claim that $y$ must select $u$ from each column to the left of the $i$th column. If $i=1$, there is nothing to prove. For $i>1$, first note that $y$ selects $u$ from the first column. Indeed, if $y$ selects $w$ in the first column, then $y$ must be at least quadratic in $w$ if $i\neq p+1$, or would have the form $v^{p-1}w$ if $i=p+1$. Now, since $y$ selects $u$ from the first column, it cannot select anything else from the first row of $P$, forcing it to select $u$ from column $2$ if $i>2$. Continuing in this fashion, we see that $y$ must select $u$ in all columns to the left of the $i$th column.

  Now, inspection of $Q$ shows that for each $u$ which $y$ selects to the right of the $i$th column, it must select a $w$ lying diagonally to the upper right. Likewise, for $y$ to select a $w$ entry, it must also select a $u$ lying diagonally to the lower left and to the right of the $i$th column. Hence, the only monomial of the desired form which appears in $q_i$ is $u^iv^{p-(i+1)}w$, and it may be created by $p-j$ different choices of the position of $w$.
\end{proof}

\begin{proof}[Proof of Lemma~\ref{lemma:rowperm}]
Write $a_i=\sum \lambda_{i,j} z_j$ for $\lambda\in\KK$.
First suppose $m=2$, so $n>2$ by assumption.  In this case, we appeal to the explicit  primary decomposition of the $2\times 2$ permanents of a $2\times n$ matrix over any field proved by Laubenbacher and Swanson \cite{laubenbacher}.  Each component is isolated in this case, and the components are as follows: there are two components that correspond to zeroing out each row; and there are ${n\choose 2}$ components that correspond to zeroing out a given $2\times 2$ permanent as well as all of the other $2(n-2)$ entries.  
It immediately follows from this classification that if at least three entries of the $1\times n$ matrix $C$ are nonzero, then $\vec a$ must be 0.

Now suppose $m>2$.  We begin by considering the permanent of the submatrix of $P$ consisting of its first $m$ columns. Inspection of monomials of the form $z_0^{i}z_1^{m-i}$ and $z_1^{i}z_2^{m-i}$
gives equations 
\begin{align*}
  \lambda_{i,0}+\lambda_{i+1,1}=0&\qquad &\textrm{for } 1\leq i \leq m-1;\\
  \lambda_{i,1}+\lambda_{i+1,2}=0&\qquad &\textrm{for } 2\leq i \leq m-1;\\
  \lambda_{m,1}+\lambda_{1,2}=0&
\end{align*}
along with   $\lambda_{2,2}=\lambda_{1,1}=\lambda_{m,0}=0$.
On the other hand, inspection of the monomials of the form $z_0^iz_1^{m-1-i}z_2$ give equations
\begin{align*}
  \overline{(m-1-i)}\lambda_{i,0}+\overline{(m-2-i)}\lambda_{i+1,1}+\lambda_{i+2,2}=0&\qquad &\textrm{for } 1\leq i \leq m-2
\end{align*}
by Lemma \ref{lemma:permterms}. Likewise, inspection of the monomial 
$z_0z_1^{m-3}z_2^2$ gives the equation
$\overline{m-2}\lambda_{1,2}+\overline{m-3}\lambda_{m,1}+\lambda_{m-1,0}$.
Since $\chara(\KK)\neq 2$, this system of equations implies that $\lambda_{i,j}=0$ for $1\leq i \leq m$ and $j={1,2,3}$.
Now, for $k=m+1,\ldots,n$, consider the first $m-1$ columns and the $k^{th}$ column of $P$.  The permanent of this $m\times m$ matrix is zero, and all of the entries in the first row are zero except possibly $a_k$.  Expanding the permanent by the first row, we have $a_k=0$.
\end{proof}

\begin{proof}[Proof of Lemma~\ref{lemma:2case}]
 First, let's prove that for $i=1,2$, each form in column $i$ is a scalar multiple of $\ell_i$.  Let $Q'_{m1}$ be the matrix obtained from $Q'$ by deleting the last row and first column.  Let $z_j$ be any form in the support of $\ell_2$, and consider the image in $\KK[z_0,\ldots,z_k]/\ell_2 \cong S=\KK[z_0,\ldots,\hat{z_j},\ldots,z_k]$ of each $r\times r$ subpermanent of $Q'$ that involves the last row of $Q'$.  By expanding along the last row, this image is $\overline{\ell_1}\cdot \perm(P)$, where $P$ 
is an $(r-1)\times (r-1)$ submatrix of $Q'_{m1}$.
        
        Since $\overline{\ell_1} \ne 0$ by assumption, $\perm(P) = 0$ as an element of $S$.  Then by Lemma~\ref{lemma:rowperm} applied over the ring $S$ to $Q'_{m1}$, the first column of $Q'_{m1}$ is zero.  In other words, every entry in the second column of $Q'$ is a scalar multiple of $\ell_2$.  Similarly, every entry in the first column of $Q'$ is a scalar multiple of $\ell_1$.

	Let ${\tilde Q}$ be the $r\times r$ submatrix consisting of the first $r-1$ rows along with the last row of ${\tilde Q}$. Factoring out $\ell_1$ and $\ell_2$ from the first two columns of $\tilde Q$, we see
$$0 = \perm({\tilde Q}) = \ell_1\ell_2\cdot(\text{a $\KK$-linear combination of the maximal subpermanents of $B$}),$$  
and in fact these maximal subpermanents are actually $\KK$-linearly independent, since, for example, they each contain a distinct term of highest total $(z_0,z_1)$-degree.  So the coefficients are zero. In other words, for $i=1,\ldots,r-1$, the $2\times 2$ permanent consisting of the $i^{th}$ row of $A'$ along with $(\ell_1~\ell_2)$ is zero.  

Similarly, for $i=r,\ldots,m-1$, let ${\tilde Q}$ be the $r\times r$ submatrix consisting of the first $r-2$ rows, the $i^{th}$ row, and the last row of ${\tilde Q}$.  Then since the uppermost $(r-2)\times(r-2)$ subpermanent of $B$ is the unique maximal subpermanent of $B$ with a nonzero monomial $z_0^{r-2}$, but $\perm({\tilde Q}) = 0$, it follows again that 
the $2\times 2$ permanent consisting of the $i^{th}$ row of $A'$ along with $(\ell_1~\ell_2)$ is zero.
We conclude that each row of $A'$ is of the form $(a\ell_1~-\!a\ell_2)$ for some $a\in \KK$.

\end{proof}

\begin{proof}[Proof of Lemma~\ref{lemma:stars}]
The proof is by induction on $p+q$ with trivial base case.  First, if $c > p-q$, then by the pigeonhole principle, there is a column of all $*$s.  Moving that column to be the rightmost one and inducting on the remaining $p\times(q-1)$ matrix, we are done.  So we may assume $c\le p-q$.  Let $z$ be the smallest number of zeroes in any column of $B$.  We claim that $z \le p-c-q+1$.  If $z\le 1$ this is clear because $p-c-q+1 \ge 1$.  Otherwise, since there are at least $q z$ zeroes in $B$, we have
$p-c \ge q z \ge q+z-1$, implying the claim.

Then we can permute rows and columns so that the last column of $B$ has $*$s for its first $p-z$ entries and zeroes for its last $z$ entries.  The claim we just proved above says precisely that the last column satisfies the hypotheses of the lemma.  Furthermore, the upper left $(p-z)\times(q-1)$ submatrix of $B$ has at most $p-z-c$ zeroes, so we are done by induction.
\end{proof}

\section{Results on compression spaces}\label{sec:compression}
In this section, we will prove a number of results on compression spaces that will be used in Section~\ref{sec:smooth} to prove results on smoothness and connectedness of our Fano schemes.  We will start by proving that the subschemes $\comp_k(s)$ of $\fd$ are actually irreducible components.  This result relies on the tangent space dimension at a general point of $\comp_k(s)$, computed in Theorem~\ref{thm:ts}.  Similarly, we will prove that when the conditions on $k$ and $s$ in Theorem~\ref{thm:tsperm} hold, then $\comp_k(\sigma,\tau)$ is a component, where $|\tau|=s$.   Finally, we will prove a theorem describing precisely the embedding in projective space of the components $\comp(s)$ of $\bF_{\kappa(s)}(\dmn)$, including a computation of their degrees.

As discussed in Section~\ref{sec:torus}, for each $k \le \kappa(s)$ there is a natural map $$\rho_k(s):\Gr(k+1,\U(s))\to\bF_k(\dmn),$$ where $\U(s)$ is the pullback to $\comp(s)$ of the universal bundle on $\bF_{\kappa(s)}(\dmn)$; the image of this map is $\comp_k(s)$.  The next theorem says that $\rho_k(s)$ is generically finite.  Since the dimension of the space $\Gr(k+1,\U(s))$ is exactly the general tangent space dimension calculated in Theorem~\ref{thm:ts}, it will follow that $\comp_k(s)$ is a component (Corollary~\ref{cor:comp}).

\begin{theorem}\label{thm:rho}
  The map $\rho_k(s)$ is generically finite.  Furthermore, it is a closed embedding if and only if 
  \begin{align*}
    k&>\kappa(s)-(m-s)\qquad&\textrm{if}\ s\neq r-1;\\
   \textrm{and } k&>\kappa(s)-(n-r+s+1)\qquad&\textrm{if}\ s\neq 0.
  \end{align*}
Furthermore, the image of $\rho_k(s)$ is smooth if and only if the above bounds on $k$ hold. 
\end{theorem}
Before proving the theorem, we will need the following lemma:

\begin{lemma}\label{lemma:intersect}
Consider compression spaces $P_1$ and $P_2\in\comp(s)$ with $P_1\neq P_2$. Write $P_1\cap P_2$ for their intersection as $\kappa(s)$-planes in $\PP^{mn-1}$. Then
\begin{align*}
    \dim (P_1\cap P_2)&\leq \kappa(s)-\min\{m-s,~s+1+n-r\}.\\
\intertext{Furthermore, if $s=0$ or $s=r-1$ then we can say which term in the minimum to take: we have}
	\dim (P_1\cap P_2)&\leq \kappa(s)-(m-s) \qquad &\textrm{if}\ s=0, \textrm{ and}\\
	\dim (P_1\cap P_2)&\leq \kappa(s)-(s+1+n-r) \qquad &\textrm{if}\ s=r-1.\\
\end{align*}
\end{lemma}

\begin{proof}
Write $a = s+1+n-r$ for convenience.  For $i=1,2$, $P_i$ can be characterized as the space of matrices that sends an $a$-dimensional subspace $V_i \subseteq \KK^n$ into an $s$-dimensional subspace $W_i \subseteq \KK^m$.  Then $P_1\cap P_2$ is the space of matrices that 
\begin{enumerate}[(i)]
  \item send $V_1$ to $W_1$;\label{cond:(i)} 
  \item send $V_2$ to $W_2;$.\label{cond:(ii)}
\end{enumerate}
As a consequence of  (\ref{cond:(i)}) and (\ref{cond:(ii)}) we also have that elements of $P_1\cap P_2$
\begin{enumerate}[(i)]
\setcounter{enumi}{2}
  \item send $V_1\cap V_2$ to $W_1\cap W_2$.\label{cond:(iii)}  
\end{enumerate}
We will show that $\dim (P_1\cap P_2)$ is maximized when $V_1 \cap V_2$ and $W_1\cap W_2$ are as large as possible.

Let $b = \dim V_1\cap V_2$.  Then (\ref{cond:(iii)}) gives $b\cdot \codim (W_1\cap W_2)$ linear conditions on the matrices in $P_i$, and (\ref{cond:(i)}) and (\ref{cond:(ii)}) each give $(a-b)(m-s)$ additional linear conditions.  So because of (\ref{cond:(iii)}), if we fix $V_1$ and $V_2,$ then choosing $W_1$ and $W_2$ to have as large intersection as possible gives $\dim (P_1\cap P_2)$ as large as possible.  Similarly, suppose we fix $W_1$ and $W_2$ and increase the dimension of $V_1\cap V_2$ by one.  This adds only $\codim (W_1\cap W_2) \le 2(m-s)$ conditions in (\ref{cond:(iii)}), but removes $2(m-s)$ conditions arising from (\ref{cond:(i)}) and (\ref{cond:(ii)}).  Thus, $\dim(P_1\cap P_2)$ attains its maximum when $V_1\cap V_2$ and $W_1\cap W_2$ are as large as possible.

The only constraint, of course, is that $P_1\ne P_2$, i.e.~we don't have both $V_1=V_2$ and $W_1=W_2$.  Therefore, $\dim(P_1\cap P_2)$ is largest if either $V_1=V_2$ and $\dim (W_1\cap W_2) = s-1$, or if $\dim (V_1\cap V_2) = a-1$ and $W_1=W_2$.  In the first case, by counting linear constraints on the points of $\PP^{mn-1}$, we have $$\dim(P_1\cap P_2) = (mn - 1) - 2(m-s) - (a-1)(m-s)= \kappa(s)-(m-s)$$
and in the second case we have
$$\dim(P_1\cap P_2) = (mn - 1) - a(m-s+1) = \kappa(s)-(s+1+n-r).$$
We conclude that for any $P_1\ne P_2$, $\dim (P_1\cap P_2)\leq \kappa(s)-\min\{m-s,~s+1+n-r\}.$

Furthermore, if $s=0$ then $\dim (W_1\cap W_2) = s-1$ is not possible, and similarly if $s=r-1$ then $\dim(V_1\cap V_2) = a-1 = n-1$ is not possible (since if $s=r-1$ then $V_1$ and $V_2$ are both $n$-dimensional, so necessarily $\dim(V_1\cap V_2)=n$ as well.)  This accounts for the stronger bounds in the second part of the Lemma in these special cases.

\end{proof}

\begin{proof}[Proof of Theorem \ref{thm:rho}]
We first show that if the lower bounds on $k$ hold, the map $\rho_k(s)$ is injective. 
Indeed, consider any two distinct points $x,y\in\Gr(k+1,\U(s))$ whose image under $\rho_k(s)$ is the same. For any value of $k$, it is clear that these points must project to distinct points $P,Q\in\comp(s)$. Furthermore, the fact that the images of $x$ and $y$ are equal implies that $P$ and $Q$ contain a common $(k+1)$-dimensional subspace. 
If we assume that the lower bounds on $k$ hold, then this is impossible by Lemma \ref{lemma:intersect}. Hence, if the lower bounds on $k$ hold, $\rho_k(s)$ is injective.

We now consider the differential $d\rho_k(s)$ of $\rho_k(s)$. We will show that this is injective everywhere if the bounds on $k$ hold.
Hence, since $\rho_k(s)$ is injective with injective differential, it is a closed embedding.
In particular, if the bounds on $k$ hold, $\comp_k(s)$ is smooth since it is isomorphic to a Grassmann bundle over a product of Grassmannians.
Furthermore, we will show that for arbitrary $k$, $d\rho_k(s)$ is injective at certain points. Hence $\rho_k(s)$ must be generically finite.

To start, we describe the map $d\rho_k(s)$. Let $Q$ be the standard compression space compressing the first $s+n-r+1$ standard basis vectors into the subspace generated by the last $s$ standard basis vectors. See (\ref{fig:compression}).  In a neighborhood of the fiber over $Q$, $\Gr(k+1,\U(s))$ trivializes as $\A^{\delta(s)}\times 
\Gr(k+1,\kappa(s)+1)$, and the tangent space of any point $q$ in the fiber decomposes as a direct sum of $$T_Q\comp(s)\cong\KK^{\delta(s)}\cong \left(\KK^{s+1+n-r}\otimes\KK^{r-s-1}\right)\oplus\left(\KK^{m-s}\otimes \KK^s\right)$$ (tangent directions on the base) and $\KK^{(k+1)(\kappa(s)-k)}$ (tangent directions in the fiber). Recall the definition of $\delta(s)$ in Equation \eqref{eqn:delta}.

Let $p\in \comp_k(s)$ be the image of $q\in\Gr(k+1,\mathcal{U}(s))$. We may represent $p$ by a matrix $P$ as in Equation \eqref{eq:matrix}. The ideal $I$ of this point in $\dmn$ is thus given by $x_{i,j}$ for $i\leq m-s$ and $j\leq s+1+n-r$, along with $\kappa(s)-k$ additional independent linear forms $f_i$ as in Theorem \ref{thm:ts}.
The tangent space $T_p\comp_k(s)$ is contained in the tangent space $T_p\bF_k(\dmn)$. As in the proof of Lemma~\ref{lemma:abc}, tangent vectors to $\bF_k(\dmn)$ at $p$ may be described by mapping the $x_{i,j}$ and $f_i$ to linear forms modulo $J$. There are no restrictions on the choice of the images of $f_i$, but the images of the $x_{i,j}$ may be constrained; denote these images by $y_{i,j}$.

Now, the summand $\KK^{(k+1)(\kappa(s)-k)}\subset T_q \Gr(k+1,\U(s))$ maps isomorphically to the subspace of $T_p\comp_k(s)$ with all $y_{i,j}=0$. Indeed, deforming $q$ within the fiber over $Q$ amounts to perturbing the $f_i$.
On the other hand, a tangent vector of the form 
$$(t\otimes u,v\otimes w)\in \left(\KK^{s+1+n-r}\otimes\KK^{r-s-1}\right)\oplus\left(\KK^{m-s}\otimes \KK^s\right)$$
maps to a tangent vector with the image $y_{i,j}$ of $x_{i,j}$ given by 
$$
y_{i,j}=(Bu)_it_j+v_i(wC)_j.
$$
Here, we have chosen a basis for $T_Q \comp(s)$ exactly as used in the proof of Theorem~\ref{thm:ts}(i): we may deform $Q$ within $\comp(s)$ by performing $(m-s)m$ independent row and $(s+1+n-r)(r-s-1)$  independent column operations.
It follows from our description of the differential $d\rho_k(s)$ that $d\rho_k(s)$ is injective at $q$ if the rows of $B$, as well as the columns of $C$, are linearly independent. This latter condition is clearly the case if either the requisite bounds on $k$ hold or if $q$ is a generic point in the fiber over $Q$. Furthermore, if the bounds on $k$ hold, then applying the $(\GL_m\times \GL_n)$-action on $\Gr(k+1,\mathcal{U}(s))$ induced by the natural action on $\comp(s)$ shows that the differential is injective everywhere.
We conclude that $\rho_k(s)$ is generically finite, and is furthermore a closed embedding if the bounds on $k$ hold.

Finally, if the bounds on $k$ fail to hold, $\comp_k(s)$ must be singular. Indeed,
we will exhibit points $P$ in $\comp_k(s)$ with tangent space dimension equal to 
\begin{equation}\label{eqn:tsdim}
\dim T_P \comp_k(s)=(k+1)(\kappa(s)-k)+(k+1)(m-s)(s+n-r+1).
\end{equation}
This exceeds the dimension of $\Gr(k+1,\U(s))$ which equals the dimension of $\comp_k(s)$. Thus, $\comp_k(s)$ is singular at these points.

Suppose that $k\leq \kappa(s)-(m-s)$ and $s\neq r-1$. Consider any point $p\in \comp_k(s)$ represented by a matrix $P$ as in Equation \eqref{eq:matrix} with the first column of $B$ identically $0$. 
We may describe tangent vectors of $\comp_k(s)$ at $p$ as above.
Now, there is a point $q\in\Gr(k+1,\U(s))$ mapping to $p$ which lies in the fiber over the standard compression space in $\comp(s)$ which compresses the subspace generated by the  standard basis vectors in positions $2,\ldots,s+n-r+2$ into the subspace generated by the last $s$ standard basis vectors.
Inspecting the differential of $\rho_k(s)$ at $q$ shows that its image contains all tangent vectors with $y_{i,j}=0$ for $j\neq 1$, that is, the $y_{i,1}$ may be arbitrary. Considering preimages of $p$ in other fibers over $\comp_k(s)$ show that all $y_{i,j}$ may be arbitrary, and hence, the tangent space dimension at $p$ is as in Equation \eqref{eqn:tsdim}.

If instead $k\leq \kappa(s)-(n-r+s+1)$ and $s\neq 0$, a similar argument involving zeroing out the first row of $C$ completes the proof.
\end{proof}

\begin{cor}\label{cor:comp}
  For each $k \le \kappa(s)$, the variety $\comp_k(s)$ is an irreducible component of $\bF_k(\dmn)$, of dimension $\delta(s)+(k+1)(\kappa(s)-k)$. If conditions \eqref{eq:smootha} and \eqref{eq:smoothb} in Theorem \ref{thm:ts} are satisfied, then every point of $\comp_k(s)$ is a smooth point of $\bF_k(\dmn)$.
\end{cor}
\begin{proof}
  Since the map $\rho_k(s)$ is generically finite by Theorem \ref{thm:rho}, the variety $\comp_k(s)$ has the same dimension as the Grassmann bundle $\Gr(k+1,\U(s))$ over $\comp(s)$. The dimension of $\comp(s)$ is $\delta(s)$ (see Equation \eqref{eqn:delta}), and the rank of $\U(s)$ is $\kappa(s)+1$, hence the dimension of $\comp_k(s)$ is $\delta(s)+(k+1)(\kappa(s)-k)$. But by Theorem \ref{thm:ts}(i), this is an upper bound on the tangent space dimension of $\bF_k(\dmn)$ at a general point of $\comp_k(s)$. Hence, $\comp_k(s)$ must be an irreducible component of $\bF_k(\dmn)$.

  Furthermore, if the bounds \eqref{eq:smootha} and \eqref{eq:smoothb} on $k$ are satisfied, it follows from Theorem \ref{thm:ts}(ii) that the tangent space dimension of any point of $\bF_k(\dmn)$ equals the dimension of $\comp_k(s)$.
\end{proof}

By the same token, for numbers $k$ and $s$ satisfying the conditions in Theorem~\ref{thm:tsperm}, the subscheme $\comp_k(\sigma,\tau)$ is actually a component of $\fp$.  

\begin{prop}\label{prop:perm}
Fix integers $k$ and $s$ with $0\leq s\leq r-1$, $k\leq \kappa(s)$.
  Let $\sigma$ and $\tau$ be subsets of $\{1,\ldots,n\}$ and $\{1,\ldots,m\}$ of sizes $ s+1+n-r$ and $s$ respectively.  
    Suppose that
\begin{align*}
s=0,~r-1\qquad &\textrm{ and }\qquad 2\leq k \leq \kappa(s),\textrm{ or}\\
1\leq s\leq r-2\qquad &\textrm{ and }\qquad 5\leq k \leq \kappa(s),
\intertext{and that}
s+1+n-r \ge 3 \quad\text{  if }\quad s\ne 0\qquad &\text{ and }\qquad m-s \ge 3\quad\text{ if }\quad s\ne r-1.
\end{align*}
Then $\comp_k(\sigma,\tau)$ is an irreducible component of $\fp$.
\end{prop}

\begin{proof}
By definition (see the discussion in Section~\ref{sec:torus}), $\comp_k(\sigma,\tau)$ is isomorphic to $\Gr(k+1,\kappa(s)+1)$ and thus has dimension $(k+1)(\kappa(s)-k)$.
Now, if the assumptions hold,  we can apply the first part of Theorem \ref{thm:tsperm}
to conclude that the tangent space dimension of $\bF_k(\pmn)$ at a general point of $\comp_k(\sigma,\tau)$ is at most $(k+1)(\kappa(s)-k)$. But this is the dimension of $\comp_k(\sigma,\tau)$, hence it must be an irreducible component of $\bF_k(\pmn)$.
 \end{proof}

The next result ensures that the compression components $\comp_k(s)$ (respectively $\comp_k(\sigma,\tau)$) of $\fd$ and $\fp$ that we have now identified are distinct.

\begin{prop}\label{prop:distinct}  Let $k\ge 1$.
\begin{enumerate}[(i)]
  \item  We have $\comp_k(s)=\comp_k(s')$ if and only if $s=s'$. \label{part:(i)}
  \item For any \label{part:(ii)}
   $\sigma,\sigma'\subset \{1,\ldots,n\}$ and $\tau,\tau'\subset\{1,\ldots,m\}$ with $$|\sigma|-|\tau|=|\sigma'|-|\tau'|=n-r+1,$$ 
   we have
   $\comp_k(\sigma,\tau)=\comp_k(\sigma',\tau')$ if and only if  $\sigma=\sigma'$ and $\tau=\tau'$.
\end{enumerate}

\end{prop}

\begin{proof}
  Part (\ref{part:(ii)}) is more or less immediate, since we can certainly find a $k$-plane in the standard compression space $\comp(\sigma,\tau)$ with points having nonzero entries in all but the $(|\tau|+1+n-m)(m-|\tau|)$ required entries.  Such a $k$-plane clearly cannot lie inside some other standard compression space $\comp(\sigma',\tau')$, since the matrices in $\comp(\sigma',\tau')$ must have some new entries that are required to be zero.

  Let us prove (\ref{part:(i)}).  For every pair of numbers $s'$ and $s$ with $0\le s'<s < r$, we will show that there is a line in an $s$-compression space that doesn't lie in any $s'$-compression space.  That will show that $\comp_1(s')$ and $\comp_1(s)$ are distinct.  Furthermore, it would follow that any $k$-plane inside that $s$-compression space that contains that line doesn't lie in any $s'$-compression space, thus showing that $\comp_k(s)$ and $\comp_k(s')$ are distinct for all $k>1$ too.  

Indeed, for any $1\leq s \leq r-1$, consider the line in $\PP^{mn-1}$ of matrices $P(z_0,z_1)$ of the form given in Equation \eqref{eq:matrix}, setting $D=0$ and setting $B = B(z_0,z_1)$ and $C = C(z_0,z_1)$.  These matrices are defined in~\eqref{eq:bc}.

We now show that this point of $\comp_1(s)$ does not lie in $\comp_1(s')$ for any $s'<s$.  In other words, we wish to show that for all subspaces $A$ of $\KK^n$ of dimension $s'+1+n-r$, there is a choice of $z_0$ and $z_1$ such that $A$ meets the kernel of the matrix $P(z_0,z_1)$ in dimension at most $n-r$.  

Indeed, the kernel of $P(z_0,z_1)$ is the $(n-r+1)$-dimensional space spanned by the rows of the matrix
$$
\left.\left(\begin{array}{c c  c c c}
  z_1^s & -z_0z_1^{s-1}&\ldots & \pm z_0^s&\\
  &&&& I_{n-r}
\end{array}
\right |
\ \mathbf{0}\ 
\right).
$$
So if $A$ does not already contain the span of the last $n-r$ vectors, we are done.  If it does, then since the rational normal curve of the first row sweeps out an $(s+1)$-dimensional space but $A$ is only $(s'+1+n-r)$-dimensional, a general choice of $z_0$ and $z_1$ will have the property that $A$ meets the kernel of $P(z_0,z_1)$ in dimension at most $n-r$, as desired.
\end{proof}

\begin{prop}\label{prop:irredperm}
If $\fp$ is nonempty, then it is is reducible.
\end{prop}
\begin{proof}
We may assume that $n>3$, since the case of $n=2$ is just {$\bF_1(P_{2,2}^2)\cong\bF_1(D_{2,2}^2)$}, which has two components, each a copy of $\PP^1$.

Suppose first that $k>1$. Then the reducibility  follows from Propositions \ref{prop:perm} and \ref{prop:distinct}. Indeed, $$\comp_k(\{1,
\ldots,n\},\{1,\ldots,r-1\})\ \textrm{and}\ \comp_k(\{1,
\ldots,n\},\{2,\ldots,r\})$$ are distinct irreducible components.

Suppose next that  $k=1$ and $r>2$. 
We will  consider subschemes $Z=\comp_1(\{1,
\ldots,n\},\tau)$ and $Z'=\comp_1(\{1,
\ldots,n\},\tau')$, where $|\tau|=|\tau'|=r-1$ and we take $\tau,\tau'$ to have smallest possible nonempty intersection, i.e.~
$|\tau\cap\tau'|=\max\{1,2(r-1)-m)\}$.
 Each of these schemes has dimension $2(\kappa(r-1)-1)$, and their intersection has codimension $2(r-2)n$ or $2(m-(r-1))n$ in $Z,Z'$. 
{Now, applying Lemma~\ref{lemma:abc} to a point in $Z$ or $Z'$ whose $r-1$ nonzero rows form the matrix $C(z_0,z_1)$ in \eqref{eq:bc} }
 shows that the tangent space dimension of $\fp$ at general points of $Z$ and $Z'$ is bounded above by $2(\kappa(r-1)-1)+\delta(r-1)$.  {Indeed, a straightforward calculation shows that the permissible matrices $A$ in Lemma~\ref{lemma:abc} are precisely those matrices whose $m-r+1$ rows are linear combinations of the $r-1$ vectors obtained from the rows of $C(z_0,z_1)$ by flipping the sign of $z_1$, yielding a total dimension of $\delta(r-1)$.}

 Now, suppose that $Y=\bF_1(P^r_{m,n})$ is irreducible. Then the codimension of $Z,Z'$ in $\fp$ is bounded above by $\delta(r-1)=(r-1)(m-(r-1))$.
Since we must have
$$\codim_{Z} Z\cap Z'\leq \codim_Y Z',$$
we get  
\begin{align*}
2(r-2)n \leq  (r-1)(m-(r-1))&\qquad\textrm{if}\qquad |\tau\cap\tau'|=1,\\
2(m-(r-1))n \leq (r-1)(m-(r-1))&\qquad\textrm{if}\qquad |\tau\cap\tau'|=2(r-1)-m.
\end{align*}
But both of these are impossible if $r>2$, so $\fp$ cannot be irreducible.

Finally, suppose that $k=1$ and $r=2$. Then $P_{m,n}^2$ is reducible by \cite{laubenbacher}, so $\bF_1(P^2_{m,n})$ must be as well. Indeed, it is immediate from the classification of minimal primary components of $P_{m,n}^2$ \cite[Theorem 4.1]{laubenbacher} that every point of $P^2_{m,n}$ is contained in a line in $P^2_{m,n}$.  So there is a surjective map from the universal bundle $\U$ over $\bF_1(P^2_{m,n})$ to $P_{m,n}^2$.  Hence if $\bF_1(P^2_{m,n})$ were irreducible, then $\U$ and $P_{m,n}^2$ would be too.
\end{proof}

The last result in this section studies the degrees of the components $\comp(s)$ of $\bF_{\kappa(s)}(\dmn)$ for $s=0,\ldots,r-1$.  

\begin{theorem}\label{thm:comp}
  The $s$-compression spaces form a subscheme $\comp(s)$ of the Fano scheme $\bF_{\kappa(s)}(\dmn)$ isomorphic to $\Gr(s+n-r+1, n)\times \Gr(m-s,m)$.  The embedding is via the Segre product of the $(m-s)$-fold Veronese of $\Gr(s+n-r+1, n)$ in its Pl\"ucker embedding, and the $(s+n-r+1)$-fold Veronese of $\Gr(m-s, m)$ in its Pl\"ucker embedding, possibly up to a linear projection.  Thus $\comp(s)$ has degree 
$$
{\delta(s)\choose{s(m-s)}}d_1 \cdot(m-s)^{(s+n-r+1)(r-s-1)}\cdot d_2 \cdot (s+n-r+1)^{s(m-s)}
$$
where  
$$d_1 ={((s+n-r+1)(r-s-1))!\over \prod_{1\le i \le r-s-1 <j\le n}(j-i)}\quad\textrm{ and }\quad
d_2 = {(s(m-s))!\over \prod_{1\le i \le s <j\le m}(j-i)}.$$
\end{theorem}

\begin{proof}
  Let $t=s+n-r+1$.  An $s$-compression space consists of the matrices that send the rowspan of some full-rank $t\times n$ matrix $A=(a_{i,j})$ to the orthogonal complement of the rowspan of a full-rank $(m-s)\times m$ matrix $B=(b_{i,j})$.   
  We would like to express the Pl\"ucker coordinates of this point of $\comp(s)$ in terms of the Pl\"ucker coordinates of $A$ and of $B$.
  
  Now, the condition that a linear map $M=(x_{i,j})$ sends $A$ to the orthogonal complement of $B$ is precisely the condition $BMA^T = 0$, that is,
 $$(BMA^T)_{i,j}=\sum_{u=1}^m\sum_{v=1}^n b_{i,u}a_{j,v}x_{u,v}=0$$
  for $1\leq i \leq m-s$ and $1\leq j \leq t$.
Thus the compression space defined by $A$ and $B$ can be represented by the $t(m-s)\times mn$ matrix $C$ whose rows are indexed by $(i,j)$ and columns by $(u, v)$, and whose $(i,j), (u,v)$ entry is $b_{i,u}a_{j,v}$.

Write $\{p_I\}_{|I|= t}$ and $\{q_J\}_{|J|=m-s}$ for the maximal minors of $A$ and $B$ respectively.  Now the maximal minors of $C$ are polynomials of bidegree $(t(m-s),t(m-s))$ in the entries of $A$ and in the entries of $B$, and by construction, they are invariant under the action of $\GL_{t}$ and $\GL_{m-s}$ on $A$ and $B$ respectively.  It follows that they are polynomials of bidegree $m-s$ and $t$ in the Pl\"ucker coordinates of $A$ and of $B$, respectively (whenever they are nonzero), since the Pl\"ucker coordinates generate the rings of $\GL$-invariants.

We do not have a general formula for the maximal minors of $C$, but can describe the minors in special cases, and this will be enough to prove the theorem. Recall that a maximal minor of $C$ corresponds to choosing exactly $t(m-s)$ entries $x_{u,v}$ of $M$. Call a choice of exactly $m-s$ entries in each of $t$ columns of $M$ {\em $A$-pure}; call a choice of exactly $t$ entries in each of $m-s$ rows of $M$ {\em $B$-pure}.
In the following, we will give a formula expressing $A$-pure and $B$-pure minors as monomials in the $p_I$ and $q_J$.

Suppose we have an $B$-pure choice of entries, which after permuting columns can be described as follows.  Let $I_1, \ldots, I_{m-s}$ be order $t$ subsets of $\{1,\ldots, n\}$.  Then the $t(m-s)$ entries of $M$ 
\begin{equation}
\{(u, v): u = 1,\ldots, m-s, v \in I_i\} \label{eqn:pure-entries}
\end{equation}
are $B$-pure.  In other words, we choose $t$ entries in each of the first $m-s$ rows of $M$ as given by the sets $I_1,\ldots,I_{m-s}$.

We claim  in Lemma \ref{lemma:pure} below that the maximal minor of $C$ that corresponds to the choice of sets $I_1,\ldots,I_{m-s}$ in the way we just described is exactly
$$\pm p_{I_1}\cdots p_{I_{m-s}}{q_{\{1,\ldots,m-s\}}}^{t}.$$
Of course, the argument applies, after permuting columns, to any $B$-pure maximal minor; and a symmetric argument, which we omit, applies to any $A$-pure maximal minor.  Call a monomial in the $p_I$'s and $q_J$'s {\em pure} if it involves a unique $p_I$ or if it involves a unique $q_J$.  Then it would follow immediately from the claim that {\em every pure monomial appears as a maximal minor of $C$.}

Now, by the above discussion, we see that the map $X=\Gr(t,n)\times \Gr(m-s,m)\to \bF_{\kappa(s)}(\dmn)$ taking a pair of subspaces $A,B$ to the compression space mapping $A$ into the orthogonal complement of $B$ is given by a linear system $L$ which is a subspace of the complete linear system $|\CO_X(m-s,t)|$. Since every pure monomial appears in $L$, it immediately follows that the map determined by $L$ is an embedding.

The degree of $\Gr(a, b)$ in its Pl\"ucker embedding is known classically as $$(a(b-a))! /\prod_{1\le i\le a<j\le b} (j-i),$$ see e.g.~\cite{mukai}. Thus, $d_1$ and $d_2$ are just the degrees of $\Gr(s+n-r+1,n)$ and $ \Gr(m-s,m)$ in their Pl\"ucker embeddings. It follows that the degree of $X$ embedded by $|\CO_X(m-s,t)|$ is given by the formula in the Theorem, since the Hilbert polynomial of a Segre product is the product of the Hilbert polynomials. 
But since $L$ is a subspace of $|\CO_X(m-s,t)|$ that also gives an embedding, the degree of $X$ embedded via $L$ is the same.
\end{proof}

\begin{lemma}\label{lemma:pure}
  The maximal minor of $C$ corresponding to the sets $I_1,\ldots,I_{m-s}$ as in \eqref{eqn:pure-entries} is
$$\pm p_{I_1}\cdots p_{I_{m-s}}{q_{\{1,\ldots,m-s\}}}^{t}.$$
\end{lemma}

\noindent We give a running illustration of the proof that follows in Example~\ref{ex:pure}.

\begin{proof}[Proof of Lemma~\ref{lemma:pure}]
The rows of the $t(m-s)\times t(m-s)$ submatrix $C'$ of $C$ in question are indexed by pairs $(i,j)$ with $i=1,\ldots,m-s$ and $j=1,\ldots,t$.  Order them lexicographically.  The columns of $C'$ are indexed by entries $(i, j)$ with $i = 1,\ldots, t, j \in I_i$.  Order these lexicographically as well; the determinant is of course preserved up to sign.  Then $C'$ can be regarded as a block matrix with $(m-s)^2$ blocks $G_{i,j}$ of size $t\times t$ 
$$C'=\left(
\begin{matrix}
G_{1,1}&\cdots&G_{1,m-s}\\
\vdots&\ddots&\vdots\\
G_{m-s,1}&\cdots&G_{m-s,m-s}
\end{matrix}
\right)
,$$
where $G_{i,j}$ is the $t\times t$ matrix
$$G_{i,j}=\left(
\begin{matrix}
b_{i,j}a_{1,I_{j,1}} &\cdots & b_{i,j}a_{1,I_{j,t}}\\
\vdots&\ddots&\vdots\\
b_{i,j}a_{t,I_{j,1}} &\cdots & b_{i,j}a_{t,I_{j,t}}
\end{matrix}
\right)
.$$
Here, $I_{j,l}$ denotes the $l^{\rm{th}}$ element of $I_j$ in order.

But then $C'$ evidently factors into the product of two matrices which can again be regarded as having $(m-s)^2$ blocks of size $t\times t$.  
The first matrix has $(i,j)$-block equal to $b_{i,j}\cdot {\mathbf{Id}}_{t}$.  In the second matrix, each diagonal block $(i,i)$ equal to the $I_i$-submatrix of $A$, and every off-diagonal block is zero.  Thus the determinant of $C'$ is
$$p_{I_1}\cdots p_{I_{m-s}}{q_{\{1,\ldots,m-s\}}}^{t}.$$
\end{proof}

\begin{ex}\label{ex:pure}
Let $r=m=n=4$ and $s=1$, so $t=2$.  A maximal submatrix $C'$ of $C$ corresponds to a choice of $6$ entries of the $4\times 4$ matrix $X$.  Suppose we pick the $B$-pure submatrix given by
$$\left(\begin{matrix}
* & * & \cdot & \cdot \\
* & \cdot & * & \cdot \\
\cdot & * & * & \cdot \\
\cdot & \cdot & \cdot & \cdot 
\end{matrix}
\right)$$
i.e.~the entries $(1,1),(1,2),(2,1),(2,3),(3,2),(3,3)$.  Then $C'$ is 
$$\bordermatrix{     &(1,1)        &(1,2)       &(2,1)       &(2,3)       &(3,2)       &(3,3)       \cr
                (1,1)&b_{11}a_{11} &b_{11}a_{12}&b_{12}a_{11}&b_{12}a_{13}&b_{13}a_{12}&b_{13}a_{13}\cr
                (1,2)&b_{11}a_{21} &b_{11}a_{22}&b_{12}a_{21}&b_{12}a_{23}&b_{13}a_{22}&b_{13}a_{23}\cr
                (2,1)&b_{21}a_{11} &b_{21}a_{12}&b_{22}a_{11}&b_{22}a_{13}&b_{23}a_{12}&b_{23}a_{13}\cr
                (2,2)&b_{21}a_{21} &b_{21}a_{22}&b_{22}a_{21}&b_{22}a_{23}&b_{23}a_{22}&b_{23}a_{23}\cr
                (3,1)&b_{31}a_{11} &b_{31}a_{12}&b_{32}a_{11}&b_{32}a_{13}&b_{33}a_{12}&b_{33}a_{13}\cr
                (3,2)&b_{31}a_{21} &b_{31}a_{22}&b_{32}a_{21}&b_{32}a_{23}&b_{33}a_{22}&b_{33}a_{23}
                }$$
and it factors as the product
$$\left(\begin{matrix}
b_{11} &  &b_{12} &  &b_{13} & \\ 
& b_{11} &  &b_{12} &  &b_{13}  \\
b_{21} &  &b_{22} &  &b_{23} & \\
&b_{21} &  &b_{22} &  &b_{23}  \\
b_{31} &  &b_{32} &  &b_{33} & \\
&b_{31} &  &b_{32} &  &b_{33} 
\end{matrix}\right)
\left(\begin{matrix}
a_{11} & a_{12} &&&&\\
a_{21} & a_{22} &&&&\\
&& a_{11} & a_{13} &&\\
&& a_{21} & a_{23} &&\\
&&&&a_{12} & a_{13} \\
&&&&a_{22} & a_{23}
\end{matrix}\right).
$$
This shows that $\det C' = p_{12}p_{13}p_{23}q_{123}^2,$
as claimed in Lemma~\ref{lemma:pure}.
\end{ex}

\section{Smoothness and Connectedness}\label{sec:smooth}
\subsection{Smoothness and connectedness for $\bF_k(\dmn)$}\label{subsec:smooth1}
Now we are equipped to describe, in the theorem restated below, which Fano schemes $\fd$ are smooth and irreducible.
{
\renewcommand{\thetheorem}{\ref{thm:smooth}}
\begin{theorem}
 Let $1\leq k <(r-1)n$. 
  \begin{enumerate}[(i)]
	  \item The Fano scheme $\bF_k(\dmn)$ is smooth if and only if $k> (r-2)n.$ \label{partone}
	  \item The Fano scheme $\bF_k(\dmn)$ is irreducible if and only if $m\neq n$ and \newline $k>(r-2)n+m-r+1$. \label{parttwo}
\end{enumerate}
\end{theorem}
\addtocounter{theorem}{-1}
}

\begin{proof}
For part (\ref{partone}), suppose that $k> (r-2)n$. Then for $0\leq s \leq r-2$, $k> \kappa(s)-(m-s-1)$. Indeed, the function $\kappa(s)-(m-s-1)$ is convex as a function of $s$, and yields $(r-2)n$ and $(r-2)m$ when evaluated at $s=r-2$ and $s=0$, respectively. Now, since 
$$\kappa(s)-(n-r+s)=\kappa(s-1)-(m-(s-1)-1)$$
it follows that conditions \eqref{eq:smootha} and \eqref{eq:smoothb} of Theorem \ref{thm:ts} hold for all $0\leq s \leq r-1$. By Corollary  \ref{cor:comp}, we may conclude that all torus fixed points of $\fd$ are smooth, so $\fd$ must be smooth.

If on the other hand $k\leq(r-2)n$, consider any point of $\bF_k(\dmn)$ corresponding to an $m\times n$ matrix of linear forms whose first $(m-r)+2$ rows contain at most one nonzero entry. This is a point of both $\comp_k(r-1)$ and $\comp_k(r-2)$, which are distinct irreducible components by Corollary \ref{cor:comp} and Proposition \ref{prop:distinct}.  Hence $\bF_k(\dmn)$ is not smooth there.

For part (\ref{parttwo}), assume $k>(r-2)n+(m-r)+1$ and $m\neq n$. Then we claim that the only compression component appearing in $\fd$ is $\comp_k(r-1).$
Indeed, we have that $\kappa(r-2)=(r-2)n+m-r+1$ and $\kappa(r-2)-\kappa(0)=(r-2)(n-m-1)\geq 0$. Since $\kappa(s)$ is a convex function, we have $k>\kappa(s)$ for $s=0,\ldots,r-2$.  Now Proposition \ref{prop:fixed} shows that all torus fixed points of $\fd$ must lie in $\comp_k(r-1)$.
Since these points are smooth by part (\ref{partone}), $\comp_k(r-1)$ is the only irreducible component of $\fd$. 
 
 On the other hand, if $k\leq (r-2)n+(m-r)+1$, then $k\le \kappa(r-2)$ so $\comp_k(r-2)$ is also a component. Likewise, if $n=m$ then both $\comp_k(0)$ and $\comp_k(r-1)$ are components since $k\le(r-1)n-1$ by assumption, and $\kappa(0)=\kappa(r-1) = (r-1)n-1$.
\end{proof}

\begin{cor}\label{cor:disjoint}
  If $k> (r-2)n$, then $\bF_k(\dmn)$ is the disjoint union of compression space components $\comp_k(s)$.
\end{cor}
\begin{proof}
  The Fano scheme $\fd$ is smooth by Theorem \ref{thm:smooth}\eqref{partone}. Let $Z$ be any irreducible component of $\fd$. By Remark \ref{rem:intersect}, $Z$ intersects a compression space component $\comp_k(s)$. But since $\fd$ is smooth, we must have $Z=\comp_k(s)$. 
\end{proof}
\begin{rem}\label{rem:pazzis}
	It also follows from the results of \cite{beasley:87a} that, under the assumption  $k\geq \kappa(r-2)$, the only components of $\bF_k(\dmn)$ are $\comp_k(s)$ for $s\in\{r-1,r-2,1,0\}$. Furthermore, \cite[Theorem 10]{pazzis} implies that if $k\geq \kappa(r-3)$, then the only components of $\bF_k(\dmn)$ are $\comp_k(s)$ for $s\in \{r-1,r-2,r-3,2,1,0\}$.
\end{rem}

We can also give a near-complete description of when $\fd$ is connected.  

\begin{theorem}\label{thm:connecteddet}
  Suppose that $1\leq k <(r-1)n$. 
\begin{enumerate}
  \item  \label{item:conone} If there is some $s=0,\ldots,r-2$ such that $k\le \kappa(s)$ and such that the conditions (\ref{eq:smootha}) and (\ref{eq:smoothb}) in Theorem~\ref{thm:ts} hold, then 
  $\bF_k(\dmn)$ is disconnected.
  \item \label{item:contwo} 
  If not, then $\fd$ is either connected or has exactly two connected components.  
Furthermore, $\fd$ is connected if $$k>\kappa(0)\quad\textrm{ or }\quad k\le \kappa(0)-(m-r+1)(r-1).$$
In particular, if $r=m$  and either condition \eqref{eq:smootha} or condition \eqref{eq:smoothb} fails to hold for every $s=0,\ldots,r-2$ such that $k\le \kappa(s)$, then $\fd$ is  connected.

\end{enumerate}
\end{theorem}
\begin{proof}

Given $s\in\{0,\ldots,r-2\}$ such that $k\le\kappa(s)$ and \eqref{eq:smootha} and \eqref{eq:smoothb} hold, we want to show that $\fd$ is disconnected.  Since $k\le \kappa(s)$, the scheme $\fd$ contains an irreducible component $\comp_k(s)$.  Since (\ref{eq:smootha}) and (\ref{eq:smoothb}) hold by assumption, all of the points of $\comp_k(s)$ are smooth by Corollary~\ref{cor:comp}. Thus $\comp_k(s)$ is itself a connected component of $\bF_k(\dmn)$.  
On the other hand, since $k\le (r-1)n -1 = \kappa(r-1)$, the scheme $\fd$ also has an irreducible component $\comp_k(r-1)$, distinct from $\comp_k(s)$ by Proposition~\ref{prop:distinct}.  Hence $\bF_k(\dmn)$ is disconnected.

We now prove part (\ref{item:contwo}) of the theorem.  Suppose instead that for every $s=0,\ldots,r-2$ such that $k\le \kappa(s)$, at least one of conditions (\ref{eq:smootha}) and (\ref{eq:smoothb}) holds.
Let us first show that the compression components $\comp_k(s)$ form either one or two connected components; 
since all irreducible components of $\fd$ meet some compression component by Remark~\ref{rem:intersect}, the same would be true for $\fd$.

Since $k \le \kappa(r-1)$, we again have that $\comp_k(r-1)$ is an irreducible component of $\fd$.  Now let $S\subseteq \{0,\ldots,r-2\}$ be the set of numbers $s$ such that $\kappa(s)-(m-s-1)\ge k$. 
Now we claim that a number $s$ belonging to $S$ means precisely that both $\comp_k(s)$ and $\comp_k(s+1)$ appear as components of $\fd$, and the two components meet. Indeed, they would intersect at  a torus fixed point of the form
$$\left(\begin{array}{c c c c c c}
\vspace{-.1cm}
0 		& \cdots & 0 		& 0 &\quad &  \cr
\vdots	& \ddots & \vdots	&\vdots   & 		 &   \cr
0 		& \cdots & 0 		&0 && \\
0 		& \cdots & 0 		& && \\
&&&&&\\
&&&&&
\end{array}\right)$$
where the upper left $(m-s)\times(s+1+n-r)$ block, along with the top upper left $(m-s-1)\times(s+2+n-r)$ block, is zero, and furthermore all but $k+1$ of the unmarked entries are zero.  (The point is that the number of blank entries in the matrix above is $\kappa(s)-(m-s-1) + 1 \ge k+1$, i.e.~we can actually fit $k+1$ independent forms into the blanks.)

Furthermore, the only compression space components appearing in $\fd$, apart from $\comp_k(r-1)$, have the form $\comp_k(s)$ and $\comp_k(s+1)$ for $s\in S$.  This is because if $k\leq \kappa(s)$ for some $0<s<r-1$ with $s\not \in S$, then since \eqref{eq:smootha} doesn't hold, \eqref{eq:smoothb} must hold.  Then we must have
$$k\le \kappa(s)-(n-r+s)=\kappa(s-1)-(m-(s\!-\!1)-1)$$
so $s-1\in S$.

Now one can check that $\kappa(s)-(m-s-1)$ is a convex function of $s$ and that it always takes a value when $s=r-2$ at least as large as that when $s=0$.  Now we have three cases: if $S$ is empty, then $\comp_k(r-1)$ is the only compression component so $\fd$ is connected.  Next, $S$ could be a single interval $S=\{a,a+1,\ldots,r-2\}$.  In that case, by our characterization of $S$, the compression components are precisely $\comp_k(a),\ldots,\comp_k(r-1)$, and they are again all connected.
Otherwise, $S$ must have the form $s=\{0,\ldots,a\} \sqcup \{b,\ldots,r-2\}$.  In this last case, the compression components are $\comp_k(0),\ldots,\comp_k(a+1), \comp_k(b),\ldots,\comp_k(r-1)$, and they form at most two connected components.

Now we prove that the extra conditions specified in \eqref{item:contwo} show that the union of the compression components is connected (and so $\fd$ is connected).  First, if $k> \kappa(0)$, then the last case above can't occur, so $\fd$ is connected. If instead $k\leq \kappa(0)-(m-r+1)(r-1)$, then $\comp_k(0)$ and $\comp_k(r-1)$ both appear, and they must intersect. Indeed, they meet at a torus fixed point in which the first  $1+n-r$ columns and first $m-(r-1)$ rows are zero. Together with the previous paragraph, this implies that $\fd$ actually has only one connected component.  

In particular, if $m=r$, then $\fd$ is always connected.  We have already proved this if $k>\kappa(0)$.  If $k\le \kappa(0)$, then by the assumption of part~\eqref{item:contwo}, condition~\eqref{eq:smootha} holds, i.e.~$k \le \kappa(0)-(m-1)= \kappa(0)-(m-r+1)(r-1),$ so we are again done by the previous analysis.
\end{proof}

Now, Theorem \ref{thm:connected}, which characterizes exactly when 
$\mathbf{F}_k(D^m_{m,n})$ is connected, follows immediately from specializing Theorem~\ref{thm:connecteddet} to the case $r=m$.

\begin{rem}\label{rem:connectedness}
  In the special case $r=m$, the proof of Theorem \ref{thm:connecteddet} implies the following statement: \emph{If $\bF_k(D_{m,n}^m)$ is disconnected, then it has an irreducible component which is its own connected component.} This statement is not always true when $r<m$. Indeed, consider the example $n=m=8$, $r=7$, $k=40$. Then $\bF_k(\dmn)$ consists only of the components $\comp_k(0),\comp_k(1),\comp_k(r-2),\comp_k(r-1)$ by  Remark \ref{rem:pazzis}. The components $\comp_k(0)$ and $\comp_k(1)$ intersect, as do $\comp_k(r-2)$ and $\comp_k(r-1)$, and there are no other intersections. Hence the Fano scheme is disconnected, but there is no irreducible connected component.
This makes it difficult to detect disconnectedness of $\bF_k(\dmn)$ without complete knowledge of its irreducible components.

In fact, we don't know if there is a situation in which the compression components form two distinct connected components, but some other component of $\fd$ connects them.  The first case that is unknown to us is $n=m=5,~k=10,$ and $ r=4.$  In this case, the compression components split into two connected components as in the proof of Theorem~\ref{thm:connecteddet}, but $k$ is not large enough to use Remark~\ref{rem:pazzis} to guarantee that no other types of components appear.
\end{rem}

We have seen that connectedness of $\fd$ can be non-monotonic with $k$ (Table~\ref{table:det}.)  But it is nevertheless true that if $k$ is sufficiently small, 
then $\fd$ is connected.
\begin{cor}\label{cor:connected}
If $$k\leq m(r-2)-\frac{ ( (n-m)-(r-2))^2}{4},$$
then the Fano scheme $\fd$ is connected.
\end{cor}
\begin{proof}
  If we pick $k$ such that $k\le \kappa(s)-(m\!-\!s\!-\!1)$ for all $s$, then we would ensure that 
  the set $S$ in the proof of Theorem~\ref{thm:connecteddet} is $\{0,\ldots,r-2\}$ and hence $\fd$ is connected.
Now,
  $$\kappa(s)-(m\!-\!s\!-\!1) = s^2 + s(n-m-r+2) + m(r-2)$$
is minimized when $s=-(n-m-r+2)/2$, when it takes on value 
$$m(r-2)-\frac{ ( (n-m)-(r-2))^2}{4}.$$
\end{proof}

\subsection{Smoothness and connectedness for $\bF_k(\pmn)$}

We can now determine exactly when $\fp$ is smooth, as stated in Theorem~\ref{thm:smoothp}.
{
\renewcommand{\thetheorem}{\ref{thm:smoothp}}
\begin{theorem}
  Let $k\geq 1$.
The Fano scheme $\bF_k(\pmn)$ is smooth if and only if $n=2$ or if $k> (r-2)n+1.$
\end{theorem}
\addtocounter{theorem}{-1}
}

\begin{proof}
  If $n=2$, then $\bF_k(\pmn)$ is either empty or the Fano scheme of lines on a smooth quadric surface, which is the union of two disjoint copies of $\PP^1$.
  Suppose instead that $n>2$ and  $k>(r-2)n+1$.  Suppose there is an $s$-compression component $\comp_k(\sigma,\tau)$ appearing in $\fp$, i.e.~suppose $k\le\kappa(s)$.  We would like to apply Theorem \ref{thm:tsperm} to show that all the points of $\comp_k(\sigma,\tau)$  are smooth points of $\fp$.  That would mean that all torus fixed points of $\fp$ are smooth by Proposition~\ref{prop:fixed}, and so $\fp$ must be smooth.
  
First, let's check that the hypotheses of Theorem~\ref{thm:tsperm} hold for our choice of $k, r,$ and $s$.  If $r=2$, then $k\ge 2$ by our assumption on $k$, and $s=0$ or $s=r-1$ are the only possibilities. On the other hand, if $r>2$, then we have $k\geq 5$ as desired. Next, we claim that $s+1+n-r\ge 3$ if $s\ne 0$, and $m-s\ge3$ if $s\ne r-1$. The first condition could only fail if $s=1$ and $r=m=n$; but in this case $\kappa(1) =(r-2)n+1 < k$, contradicting that $k \le \kappa(1)$.  Similarly, the second condition could only fail if $s=r-2$ and $r=m$, but in this case $\kappa(r-2) = (r-2)n+1 < k$, again a contradiction.

Next, we always have that for $0\leq s \leq r-2$, $k> \kappa(s)-(m-s-2)$. Indeed, the function $\kappa(s)-(m-s-2)$ is convex as a function of $s$, and yields $(r-2)n+1$ and $(r-2)m+1$ when evaluated at $s=r-2$ and $s=0$, respectively. Furthermore, since 
$$\kappa(s)-(n-r+s-1)=\kappa(s\!-\!1)-(m-(s\!-\!1)-2),$$
it follows that conditions \eqref{eq:smoothap} and \eqref{eq:smoothbp} of Theorem \ref{thm:tsperm} hold for all $0\leq s \leq r-1$. 

Hence, we can apply Theorem \ref{thm:tsperm} to all compression components appearing in $\fp$. 
By Theorem \ref{thm:tsperm}, the tangent space dimension at any torus fixed point of $\fp$ which is an $s$-compression space is bounded above by $(k+1)(\kappa(s)-k)$. But by  Proposition \ref{prop:perm}, this point is in a compression space component of exactly this dimension. We conclude that all torus fixed points of $\fp$ are smooth, so $\fp$ must be smooth.

Suppose instead that $k\leq (r-2)n+1$ and $n\neq 2$; we must show that $\fp$ is singular. First, let us assume that $k>1$. Then by Proposition \ref{prop:perm}, $\comp_k(r-1)$ is an irreducible component and has dimension $(k+1)(\kappa(r-1)-k)$. Consider a point $P$ of $\comp_k(r-1)$ corresponding to an $m\times n$ matrix whose first $m-r+2$ rows have the form

\begin{equation}\label{matrix:singular}
\left(\begin{array}{c c c c c}
  0&\cdots& 0 &0 &0\\
  \vdots& \ddots & \vdots & \vdots &\vdots\\
  0 & \cdots & 0 & 0& 0\\
  0 & \cdots & 0 & z_0& z_1\\
\end{array}\right),
\end{equation}
  that is, the first $m-r+1$ rows are zero, and in the $(m-r+2)^{nd}$ row, all entries vanish except the two on the right. Now, perturbing the relations among the entries of $P$ not in the first $m-r+1$ rows gives $(k+1)(\kappa(r-1)-k)$ tangent directions as in the proof of Lemma \ref{lemma:abc}. However, we may also perturb $P$ by changing its first $m-r+2$ rows to
$$
\left(\begin{array}{c c c c c}
  0&\cdots& 0 &0 &0\\
  \vdots& \ddots & \vdots & \vdots &\vdots\\
  0 & \cdots & 0 & \varepsilon z_0& -\varepsilon z_1\\
  0 & \cdots & 0 & z_0& z_1\\
\end{array}\right)
  $$
to get an additional tangent vector. Hence $\fp$ 
is not smooth at this point.

To conclude, we must deal with the case $k=1$, still supposing that $n\neq 2$ and $k\leq (r-2)n+1$. 
First, suppose that $r>2$ and consider the point $P\in\comp_1(r-1)$ corresponding to the 
$m\times n$ matrix whose bottom row is $(\begin{array}{c c c c c}z_0&z_1&0&\cdots & 0\end{array})$, with $0$ entries everywhere else. Then by Lemma \ref{lemma:abc}, the dimension of the tangent space to $\fp$ at $P$ is $2(mn-2)$. Consider instead the point 
$P'\in\comp_1(r-1)$ corresponding to the 
$m\times n$ matrix 
$$
\left(\begin{array}{c}
A\\
C
\end{array}\right)
$$
with $A$ an $(m-r+1)\times n$ block of zeroes and $C=C(z_0,z_1)$ (for $s=r-1$), see \eqref{eq:bc}.
A straightforward calculation with Lemma \ref{lemma:abc} shows that $\dim T_{P'} \fp<2(mn-2)$. Hence, $\fp$ must be singular at $P$.

Finally, we now assume that  $k=1$ and $r=2$. Let $P\in\fp$ correspond to the $m\times n$ matrix of \eqref{matrix:singular}. Now, there is a first-order deformation of this linear space given by perturbing  $P$ to 
$$
\left(\begin{array}{c c c c c}
  0&\cdots& 0 &0 &0\\
  \vdots& \ddots & \vdots & \vdots &\vdots\\
  0 & \cdots & 0 & \varepsilon z_0& -\varepsilon z_1\\
  0 & \cdots & \varepsilon z_0 & z_0& z_1\\
\end{array}\right).
  $$
This cannot be lifted to second order:
there are obstruction terms $\epsilon^2z_0^2$ and $-\epsilon^2z_0z_1$ which cannot simultaneously be canceled out.
Indeed, the only possibility for canceling either obstruction is by including an order two term at the upper left of the lower $2\times 3$ block of the above matrix, and these terms do not agree.  
Hence, $\bF_1(\pmn)$ is not smooth at $P$.
\end{proof}

\begin{cor}\label{cor:permk}
  If $k> (r-2)n+1$, then all components of $\bF_k(\pmn)$ are of the form $\comp_k(\sigma,\tau)$.
  \end{cor}
\begin{proof}
  If $k>(r-2)n+1$, then by Theorem \ref{thm:smoothp}, $\fp$ is smooth. By the arguments in the proof of Theorem \ref{thm:smoothp} above, each $\comp_k(\sigma,\tau)$ is actually an irreducible component. Since any irreducible component of $\fp$ intersects some $\comp_k(\sigma,\tau)$, all components must have this form.
\end{proof}

Our result on the connectedness of $\fp$ is similar to the case for $\fd$:
\begin{theorem}\label{thm:connectedp}
  Suppose that $1\leq k <(r-1)n$. 
\begin{enumerate}
	\item \label{item:dconp}  
	Suppose there is some $s=0,\ldots,r-1$ satisfying the conditions of Theorem~\ref{thm:tsperm} and satisfying~\eqref{eq:smoothap} and~\eqref{eq:smoothbp} as well.  Then $\fp$ is disconnected.

  \item \label{item:conp}Conversely, $\bF_k(\pmn)$ is connected if $k\leq \max\{\kappa(0),\kappa(r-2)\}$ and if:
  \begin{enumerate}
	  \item  For each integer $s$ with $0< s<r-1$ satisfying $k\leq \kappa(s)$, we have 
	    $k\leq \kappa(s)-\min\{m-s-1,n-r+s\};$ and
\item  If $k\leq \kappa(0)$, then $k\leq \kappa(0)-(m-r+1)(r-1)$. 
\end{enumerate}
\end{enumerate}
\end{theorem}

\begin{proof}[Proof of Theorem \ref{thm:connectedp}]
	If the hypotheses in (\ref{item:dconp}) are met, then by Theorem \ref{thm:tsperm} and Proposition \ref{prop:perm}, $\comp_k(\sigma,\tau)$ is an irreducible component whose points are all smooth points of $\fp$, for any $\sigma\subset\{1,\ldots,n\}$, $\tau\subset\{1,\ldots,m\}$ with $|\tau|=s$ and $|\sigma|=s+1+n-r$. Each of these irreducible components $\comp_k(\sigma,\tau)$ is thus a connected component of $\fp$, so this Fano scheme is not connected.
	
	Suppose on the other hand that the hypotheses in (\ref{item:conp}) hold.
  Then as in the proof of Theorem \ref{thm:connecteddet}, 
  we may connect any irreducible component of $\fp$ to some $\comp_k(\sigma,\tau)$ where $|\tau|=r-1$. Hence, to show connectedness, we must
  simply connect all subschemes of the form $\comp_k(\sigma,\tau)$ with $|\tau|=r-1$.

  Now, if $k\leq \kappa(0)$, then we may consider the subscheme $\comp_k(\{1,\ldots,n-r+1\},\{\})$. By assumption (2b), this clearly intersects every subscheme of the form $\comp_k(\sigma,\tau)$, where $|\tau|=r-1$. Likewise, if $k\leq \kappa(r-2)$, we may assume $r>2$ for otherwise we are done by the previous case.  Then by assumption (2a), we have $$k\le \kappa(r-2)-\min\{m-r+1,n-2\}\le \kappa(r-2)-(m-r+1).$$  Then we may connect subschemes of the form $\comp_k(\sigma,\tau)$, where $|\tau|=r-1$, via subschemes of the form $\comp_k(\sigma',\tau')$, where $|\tau'|=r-2$. Hence, in both cases, $\bF_k(\pmn)$ is connected.
\end{proof}
\begin{cor}\label{cor:connectedp}
If $$k\leq m(r-2)-\frac{ ( (n-m)-(r-2))^2}{4},$$
then the Fano scheme $\fp$ is connected.
\end{cor}
\begin{proof}
Similarly to the proof of of Corollary~\ref{cor:connected}, the bound on $k$ implies that $k \le \kappa(s)-(m-s-1)$ for all $s$.  Thus all compression components $C_k(\sigma,\tau)$ appear.  Further, just as in the proof of Theorem~\ref{thm:connectedp}, for each $s=0,\ldots,r-2$, the components of the form $\comp_k(\sigma,\tau)$ with $|\tau|=s+1$ are connected via the components of the form $\comp_k(\sigma',\tau')$ with $|\tau'| = s$.  
\end{proof}

\section{Fano schemes of lines}\label{sec:lines}
To start out this section, we will prove Theorem \ref{thm:lines}, giving a complete description of the components of $\bF_1(\dmn)$, expanding on \cite[Corollary 2.2]{eisenbud:88a}.

{
\renewcommand{\thetheorem}{\ref{thm:lines}}
\begin{theorem}
  The Fano scheme $\bF_1(\dmn)$ has  exactly $r$ irreducible components, of dimensions
  $$
\delta(s)+2(\kappa(s)-1) \qquad \text{for }0\leq s \leq r-1.
  $$
In particular, if $m=n$, then each irreducible component of $\bF_1(D_{n,n}^r)$ has dimension $(n-r)(r-2)+2nr-n-5.$  If $r>2$, then all components intersect pairwise.  Furthermore, if $r = m=n$, then $\bF_1(D_{n,n}^n)$ is a reduced local complete intersection.
\end{theorem}
\addtocounter{theorem}{-1}
}

\begin{rem}
To clarify, we are claiming in Theorem~\ref{thm:lines} that $\bF_1(\dmn)$ has exactly $r$ {\em minimal} primary components; there may well be embedded components in addition to these.  On the other hand, in the hypersurface case $r=m=n$, our theorem implies that $\bF_1(\dmn)$ is reduced, so in that case there can't be any embedded components.
\end{rem}
\begin{proof}[Proof of Theorem \ref{thm:lines}]
  By considering the Kronecker canonical form of a pencil of matrices \cite[XII.4]{gantmacher}, it follows that every point of $\bF_1(\dmn)$ is contained in a subscheme of the form $\comp_1(s)$ for $s=0,\ldots,r-1$. These form exactly $r$ distinct  components of $\bF_1(\dmn)$ of the desired dimension  by Corollary \ref{cor:comp} and Proposition \ref{prop:distinct}. 
  In particular, if $m=n$, then each of the $r$ components has dimension $\delta(s)+2(\kappa(s)-1) = (n-r)(r-2)+2nr-n-5.$
  It is also easy to see that unless $r=2$, any two components $\comp_1(s)$ and $\comp_1(s')$ intersect at a torus fixed point that has all zero entries in its upper left $(m-s)\times (s+1+n-r)$ block as well as in its upper left $(m-s')\times (s'+1+n-r)$ block.
  
  If  $m=n=r$, then $\dmn$ is a hypersurface of degree $n$, and $\bF_1(D^n_{n,n})$ is the zero locus of a global section of a rank $n+1$ vector bundle on $\Gr(2,n^2)$, see e.g. \cite[Proposition 8.4]{3264}. Moreover, each of the $n$ components $\bF_1(\dmn)$ has dimension $2n^2-n-5$, implying that $\bF_1(D^n_{n,n})$ has codimension $n+1$ in $\Gr(2,n^2)$. Therefore, $\bF_1(\dmn)$ is a local complete intersection. Furthermore, it must be reduced, since it is a local complete intersection and is generically reduced by Theorem \ref{thm:ts}.
\end{proof}
\begin{rem}\label{rem:expp}
	We would like to prove a similar result for the Fano scheme $\bF_1(\pmn)$, especially in the case $r=m=n$. A modification of the argument used to prove Theorem \ref{thm:tsperm} may be used to show that for a general point $\eta$ of any $\comp_1(\sigma,\tau)$,  
	$$\dim T_\eta\bF_1(\pmn)\leq \delta(s)+2(\kappa(s)-1).$$
  In the special case $r=m=n$, the right hand side again simplifies to $2n^2-n-5$, which, as above is the expected dimension of $\bF_1(P_{n,n}^n)$. To complete the argument that $\bF_1(P_{n,n}^n)$ has the expected dimension (and is thus a reduced local complete intersection) one would need to show that 
every irreducible component of $\bF_1(P_{n,n}^n)$ contains one of the subschemes $\comp_1(\sigma,\tau)$.

  Our computer calculations show that $\bF_1(P_{n,n}^n)$ indeed has the expected dimension for $n=3$ and $4$.  We conjecture that in fact $\bF_1(P_{n,n}^n)$ has the expected dimension for all $n$. See Conjecture \ref{conj:expected}.

\end{rem}
\begin{prop}\label{prop:deglines}
  The Fano scheme $\bF_1(D_{n,n}^n)$ has degree equal to 
  $$\int_{\Gr(2,n^2)} c_{n+1} (\Sym^n S^*)\cdot \mcL^{2n^2-n-5},$$
  where $S$ denotes the tautological rank $2$ subbundle and $\mcL$ the tautological line bundle of $\Gr(2,n^2)$. For $n\leq 6$, these values are recorded in Table \ref{table:deg}. 
\end{prop}

\begin{table}
 \begin{tabular}{|c| r|}
   \hline
    $n$ & Degree of $\bF_1(D_{n,n}^n)$\\
    \hline
    $2$ & $4$\\
   $3$ & $2754$\\
    $4$ & $97943936$\\
    $5$ & $91842552457500$\\
    $6$ &$1905481100678765027040$\\
\hline
\end{tabular}

\vspace{.5cm}

\caption{Degrees of Fano schemes of lines}\label{table:deg}
\end{table}

\begin{proof}
  By Theorem \ref{thm:lines},  $\bF_1(D_{n,n}^n)$ has the expected codimension of $n+1$ in $\Gr(2,n^2)$. Hence,  \cite[Proposition 8.4]{3264} implies that the Chow class of $\bF_1(D_{n,n}^n)$ in the Chow ring of $\Gr(2,n^2)$ is just $c_{n+1} (\Sym^n S^*)$, and the claim follows.

  The degrees in Table~\ref{table:deg} may now be explicitly computed using Schubert calculus. We do this using the {\tt{Schubert2}}  package of \emph{Macaulay2} \cite{M2}:
  \begin{verbatim}
loadPackage "Schubert2"
G = flagBundle({2,n^2-2});
(S,Q)=G.Bundles;
c= chern(n+1,symmetricPower(n,dual S));
L=chern_1 tautologicalLineBundle G;
integral (c*L^(2*n^2-n-5))
  \end{verbatim}
\end{proof}

\begin{rem}
  We cannot use these methods to compute the degree of $\bF_k(D_{n,n}^n)$ for $k>1$. Indeed,  the dimension of $\bF_k(D_{n,n}^n)$ is larger than the expected dimension, as a quick computation shows.
\end{rem}

\begin{rem}\label{rem:degperm}
	If one could show that $\dim \bF_1(P_{n,n}^n)=\dim \bF_1(D_{n,n}^n)$, then the above degree computations hold for the Fano scheme $\bF_1(P_{n,n}^n)$ as well. In particular, they hold for the cases $n=3$ and $4$, as in Remark \ref{rem:expp}.
\end{rem}
  
  \section{Fano schemes for $3\times 3$ matrices}\label{sec:3}
  \subsection{The Fano schemes $\bF_k(D_{3,3}^3)$}
  Using the results of the previous sections, we may glean quite a bit of information about $\bF_k(D_{3,3}^3)$, $1\leq k\leq 5$.  It is never irreducible, it is smooth if and only if $k=4$ or $5$, and it is disconnected again if and only if $k=4$ or $5$. Furthermore, the degree of $\bF_5(D_{3,3}^3)$ is $18$, and the degree of $\bF_1(D_{3,3}^3)$ is $2754$.

  We may use the results of \cite{atkinson:83a} to actually describe all non-embedded irreducible components of $\bF_k(D_{3,3}^3)$. Let $\comp^*$ be the $\GL_3\times \GL_3$-orbit closure of a general $3\times 3$ anti-symmetric matrix in $\bF_2(D_{3,3}^3)$. It follows from 
  \cite{atkinson:83a}
  that any point of $\bF_k(D_{3,3}^3)$ is either a subspace of a compression space, or in the case $k=2$, a point of $\comp^*$, see also \cite[Theorem 1.1]{eisenbud:88a}.
  
  \begin{prop}\label{prop:threetimesthree}
    The irreducible components of $\bF_k(D_{3,3}^3)$ are exactly the ones described in Table \ref{table:comps}.
\end{prop}

  \begin{table}

    \begin{tabular}{|c| c| c|}
     \hline
      &\multicolumn{2}{c|}{$\bF_1(D_{3,3}^3)$}\\
      \cline{2-3}
      & Dimension & Is smooth?\\
      \hline
      $\comp_1(0)$ & $10$ & No\\
      $\comp_1(1)$ & $10$ & No\\
      $\comp_1(2)$ & $10$ & No\\
      \hline
    \end{tabular}
    
\vspace{.5cm}

    \begin{tabular}{|c| c| c|}
     \hline
      &\multicolumn{2}{c|}{$\bF_2(D_{3,3}^3)$}\\
      \cline{2-3}
      & Dimension & Is smooth?\\
      \hline
      $\comp_2(0)$ & $11$ & No\\
      $\comp_2(1)$ & $10$ & No\\
      $\comp_2(2)$ & $11$ & No\\
      $\comp^*$ & $8$ & Yes\\
      \hline
    \end{tabular}

\vspace{.5cm}

    \begin{tabular}{|c| c| c|}
     \hline
      &\multicolumn{2}{c|}{$\bF_3(D_{3,3}^3)$}\\
      \cline{2-3}
      & Dimension & Is smooth?\\
      \hline
      $\comp_3(0)$ & $10$ & Yes\\
      $\comp_3(1)$ & $8$ & Yes\\
      $\comp_3(2)$ & $10$ & Yes\\
      \hline
    \end{tabular}

\vspace{.5cm}
    \begin{tabular}{|c| c| c|}
     \hline
      &\multicolumn{2}{c|}{$\bF_4(D_{3,3}^3)$}\\
      \cline{2-3}
      & Dimension & Is smooth?\\
      \hline
      $\comp_4(0)$ & $7$ & Yes\\
      $\comp_4(1)$
      & $4$ & Yes\\
      $\comp_4(2)$ & $7$ & Yes\\
      \hline
    \end{tabular}

\vspace{.5cm}

    \begin{tabular}{|c| c| c|}
     \hline
      &\multicolumn{2}{c|}{$\bF_5(D_{3,3}^3)$}\\
      \cline{2-3}
      & Dimension & Is smooth?\\
      \hline
      $\comp_5(0)$
      & $2$ & Yes\\
      $\comp_5(2)$
      & $2$ & Yes\\
      \hline
    \end{tabular}

\vspace{.5cm}
    
    \caption{Irreducible components of $\bF_k(D_{3,3}^3)$ for $1\le k\le 5$.}\label{table:comps}
  \end{table}

  \begin{proof}
    That these are exactly the components follows from the above discussion, Corollary \ref{cor:comp}, and Proposition \ref{prop:distinct}. The dimension calculations, with the exception of $\comp^*$, follow from Corollary \ref{cor:comp} as well. 
    The dimension of $\comp^*$ follows from a straightforward calculation of the stabilizer of $\GL_3\times \GL_3$ at the point corresponding to a general antisymmetric matrix, which has dimension $10$.

    Regarding smoothness, the components $\comp_k(s)$ are smooth if and only if $k\geq 3$ by Theorem \ref{thm:rho}.
To show that the component $\comp^*$ is smooth, it suffices to show that it is smooth at its torus fixed points.
Now, we claim that the only fixed points in $\comp^*$ are $S_3\times S_3$-equivalent to the point $P$ in Example \ref{ex:nonreduced} below. Indeed, let $Q$ be a general $3\times 3$ antisymmetric matrix.
A straightforward calculation shows that for any $A,B\in \GL_3$, the space of linear forms spanned by the first row of $A\cdot Q \cdot B$ is at most two-dimensional. Hence, the $\GL_3\times \GL_3$ orbit of $Q$ does not intersect the Pl\"ucker chart containing the torus fixed point
$$
P'=\left(\begin{array}{c c c}
*&*&*\\
0&0&0\\
0&0&0
  \end{array}\right),
$$
so the orbit closure $\comp^*$ does not contain this fixed point, or any fixed point $S_3\times S_3$-equivalent to it.
Furthermore, $\comp^*$ cannot contain a fixed point $S_3\times S_3$-equivalent to 
$$
P''=\left(\begin{array}{c c c}
*&*&0\\
0&0&*\\
0&0&0
  \end{array}\right).
$$
Indeed, $P''$ is in the same $GL_3\times GL_3$ orbit as 
$$
\left(\begin{array}{c c c}
z_0&z_1&z_2\\
0&0&z_2\\
0&0&0
  \end{array}\right),
$$
whose $T$-orbit closure contains $P'$.
The same argument excludes any torus fixed points $S_3\times S_3$-equivalent to $P'^T$ or $P''^T$.  Hence, all fixed points in $\comp^*$ are of the desired form.

Now, an explicit calculation as in the example below shows that $\comp^*$ is smooth at $P$. Hence, $\comp^*$ is smooth at all its torus fixed points, and thus smooth.
  \end{proof}
  \begin{rem}
    Even though all irreducible components of $\bF_3(D_{3,3}^3)$ are smooth, the Fano scheme itself is not smooth, since these components have non-empty intersection.
  \end{rem}
  \begin{ex}[A formal neighborhood in $\bF_2(D_{3,3}^3)$]\label{ex:nonreduced}  
We will analyze a neighborhood of the torus fixed point 
    \begin{equation*}
 P= \left(\begin{array}{c c c}
*&*&0\\
*&0&0\\
0&0&0
  \end{array}\right)
\end{equation*}
in $\bF_2(D_{3,3}^3)$. In this case, it is advantageous  to consider a \emph{formal} neighborhood, since the defining equations are easier to calculate.

First, let us assume that $\chara \KK=0$. Using \emph{Macaulay2} \cite{M2} and the package {\tt VersalDeformations} \cite{ilten:versal}, we calculate a formal neighborhood of this fixed point:
\begin{verbatim}
loadPackage "VersalDeformations";     
A=QQ[x_1..x_9];
M=genericMatrix(R,3,3);  
S=A/ideal det M;                    
J=ideal(x_3,x_5..x_9);
(F,R,G,C)=localHilbertScheme(gens J); 
\end{verbatim}
Primary decomposition of the obstruction equations gives two $11$-dimensional components, one $10$-dimensional component, four $8$-dimensional components, and one $7$-dimensional component.
\begin{verbatim}
decomp=primaryDecomposition(sub(ideal sum G,QQ[gens ring G_0]));
apply(decomp,i->dim i)
\end{verbatim}
Closer inspection shows that the $11$- and $10$-dimensional components are non-embedded, as is one $8$-dimensional component.  Furthermore, these components are all smooth.
The other three $8$-dimensional components are embedded in the $10$-dimensional component, and the $7$-dimensional component is embedded in the two $11$-dimensional components. 

We draw two conclusions. First, $\bF_2(D_{3,3}^3)$ is non-reduced. Second, all irreducible components of $\bF_2(D_{3,3}^3)$ (with their reduced structure) are smooth at $P$. In particular, $\comp^*$ is smooth at $P$. 

Even if $\chara \KK >0$, the above calculation shows that $\comp^*$ is smooth at $P$. Indeed, it still follows from the calculation that $\bF_2(D_{3,3}^3)$ contains a  $10$-dimensional subscheme $Z$ contained in $\comp^*$ and smooth at $P$. But since $\dim \comp^*=10$, this is just $\comp^*$.
\end{ex}

\subsection{The Fano Schemes $\bF_k(P_{3,3}^3)$}\label{sec:threetimesthreep}
Using the results of the previous sections, we may also say quite a bit about $\bF_k(P_{3,3}^3)$ for $1\leq k\leq 5$.  It is never irreducible, and it is smooth if and only if $k=5$. It is disconnected if $k=5$ and connected if $k\leq 3$. By the discussion below, we will see that it is disconnected if $k=4$. Furthermore, $\bF_5(P_{3,3}^3)$ consists just of  $6$ points, and hence has degree $6$, while by Remark \ref{rem:degperm}, the degree of $\bF_1(P_{3,3}^3)$ is $2754$.

{\bf 4-planes on the $3\times 3$ permanental hypersurface.}  As an example of the rich geometry that can occur, we will now give a detailed description of $\bF_4(P_{3,3}^3)$. We will do this by looking at the local structure structure of  $\bF_4(P_{3,3}^3)$ around torus fixed points.
For simplicity, we assume $\chara \KK =0$, although we expect the calculation to hold in arbitrary characteristic.

Let's start by considering the fixed point 
$$P=\left(\begin{array}{c c c}
0&0&*\\
0&0&*\\
*&*&*
\end{array}\right).
$$
There are nine fixed points in the $S_3\times S_3$ orbit of $P$. Now, explicit calculation on the corresponding Pl\"ucker chart shows that around $P$, $\bF_4(P_{3,3}^3)$ is cut out by the equations $r_1r_2=s_1s_2=0$ in $\Spec \KK[r_1,r_2,s_1,s_2]$, with a parametrization given by
$$\left(\begin{array}{c c c}
s_1z_0+r_1z_2-r_1s_1z_4&s_2z_0-r_1z_3+r_1s_2z_4&z_0\\
-s_1z_1+r_2z_2+r_2s_1z_4&-s_2z_1-r_2z_3-r_2s_2z_4&z_1\\
z_2&z_3&z_4
\end{array}\right).
$$
On this chart, $\bF_4(P_{3,3}^3)$ thus decomposes into four copies of $\A^2$. 

Each copy of $\A^2$ compactifies in $\bF_4(P_{3,3}^3)$ to a $\PP^1\times \PP^1$, as can be seen by considering transition maps to the charts containing the other fixed points in the $S_3\times S_3$ orbit of $P$.
In fact, one can parametrize these components explicitly. 
For example, for the component given locally by $r_2=s_2=0$, we can identify a point $((u_0:u_1),(v_0:v_1))$ with the space of matrices that are orthogonal to the four matrices
\begin{equation*}
\left(\begin{array}{c c c}
0&0&0\\
  0&1&0\\
0&0&0\\
\end{array}\right)
\quad
\left(\begin{array}{c c c}
0&0&0\\
u_0&0&u_1\\
0&0&0
\end{array}\right)
\quad
\left(\begin{array}{c c c}
0&v_0&0\\
0&0&0\\
0&v_1&0
\end{array}\right)
\quad 
\left(\begin{array}{c c c}
u_0v_0&0&-u_1v_0\\
0&0&0\\
-u_0v_1&0&u_1v_1
\end{array}\right)
\end{equation*}
\noindent with respect to the entrywise inner product.
On the affine chart of this $\PP^1\times \PP^1$ given by $u_0\neq 0$, $v_0\neq 0$, there is an isomorphism with the $r_2=s_2=0$ copy of $\A^2$ from above, gotten by setting $s_1=u_1/u_0$ and $r_1=v_1/v_0$.
  The other three standard affine charts of this $\PP^1 \times \PP^1\subset \bF_4(P_{3,3}^3)$ are just copies of $\A^2$ containing other fixed points in the $S_3\times S_3$ orbit of $P$.

The $S_3\times S_3$ action on rows and columns gives a total of nine copies of $\PP^1\times \PP^1$; these are exactly the irreducible components containing a fixed point equivalent to $P$. The previous parametrization shows that they are all embedded in a linear subspace of $\Gr(5,9)$ by $\CO(2,2)$, and hence each has degree $8$.
Together, they form a connected component of $\bF_4(P_{3,3}^3)$. Its dual intersection complex, drawn on the fundamental domain of a torus, is pictured in Figure \ref{fig:complex1}. The fixed points, corresponding to the nine squares, are labeled.

\begin{figure}
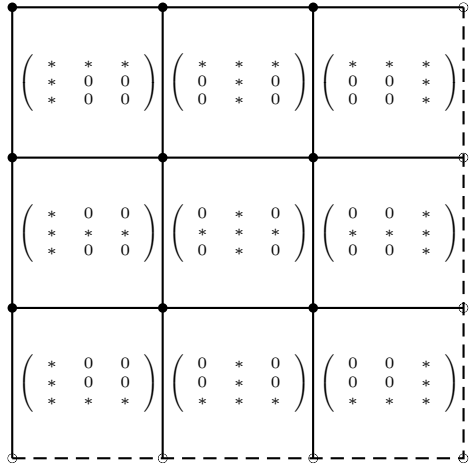

  \begin{center} \intersectioncompa
  \end{center}
  \caption{Dual intersection complex for the first connected component of $\bF_4(P_{3,3}^3)$}\label{fig:complex1}
\end{figure}

Next, consider the fixed point 
$$Q=\left(\begin{array}{c c c}
0&0&0\\
0&*&*\\
*&*&*
\end{array}\right).
$$
There are 18 fixed points in the $S_3\times S_3$ orbit of $Q$, and after transposition, one gets 
18 more fixed points. Together with the orbit of $P$ above, this covers all fixed points in $\bF_4(P_{3,3}^3)$.

Explicit calculation on the Pl\"ucker chart containing $Q$ shows that $Q$ is contained in $4$ components of $\bF_4(P^3_{3,3})$. One component is just the component $\comp_4(\{1,2,3\},\{2,3\})$, which is isomorphic to $\Gr(5,6)\cong \PP^5$, linearly embedded and hence of degree $1$.
There is an  embedded component which is just a fat point supported on $Q$. 
The other two components are $2$-dimensional, locally cut out by the equation $t_1t_2=0$ in 
in 
 $\Spec \KK[t_1,t_2,t_3]$, with a parametrization given by
$$\left(\begin{array}{c c c}
-t_1z_2&t_2t_3z_3+t_1z_3-t_3z_0&-t_2t_3z_4+t_1z_4+t_3z_1\\
-t_2z_2&z_0&z_1\\
z_2&z_3&z_4
\end{array}\right).
$$

By again considering transition maps to other charts, one sees that both of these components compactify in $\bF_4(P_{3,3}^3)$ to $\F_1$, the first Hirzebruch surface.

As above, we can parametrize these components explicitly. Consider the Cox ring $\KK[u_0,u_1,v,w]$ of $\F_1$, with the degrees of $u_0,u_1,v,w$ given by the columns of the matrix 
$$
\left(\begin{array}{c c c c}
 1&1&1&0\\
 0&0&1&1
\end{array}\right).
$$
For the component given locally by $t_2=0$, we can identify a point $(u_0:u_1:v:w)$ with the space of matrices orthogonal to the four matrices
\begin{equation*}
\left(\begin{array}{c c c}
0&0&0\\
1&0&0\\
0&0&0\\
\end{array}\right)
\quad
\left(\begin{array}{c c c}
u_0&0&0\\
0&0&0\\
u_1&0&0
\end{array}\right)
\quad
\left(\begin{array}{c c c}
0&u_0w&0\\
0&v&0\\
0&u_1w&0
\end{array}\right)
\quad 
\left(\begin{array}{c c c}
0&0&u_0w\\
0&0&-v\\
0&0&-u_1w
\end{array}\right)
\end{equation*}
\noindent with respect to the entrywise inner product.
A parametrization of the other component is gotten by an appropriate permutation of the above.
It follows from this parametrization  that both copies of $\F_1$ are embedded by a divisor of type $(3,2)$ in the given basis of the Picard group, which is simply the anticanonical class. In particular, each of these components has degree $8$. 

By letting $S_3\times S_3$ act, we get a total of nine copies of $\F_1$, each of which contains four fixed points in the $S_3\times S_3$ orbit of $Q$. The dual intersection complex is 
the simplicial complex which is the union of the boundary complexes of three disjoint 2-simplices, with intersections taking place in torus invariant divisors with self-intersection zero. The  $S_3\times S_3$ action also gives us $3$ copies of $\PP^5$, and each $\F_1$ intersects exactly one of these in a line and the other two in a point.
Together, these form a second connected component of $\bF_4(P_{3,3}^3)$. 

Finally, transposing all matrices gives a third connected component, isomorphic to the second. We conclude that $\bF_4(P_{3,3}^3)$ is neither reduced nor connected, though in their reduced structures, every irreducible component is smooth.
Summing this all up, we have the following:

\begin{table}
    \begin{tabular}{c  c}

    \begin{tabular}{|c| c| c|}
     \hline
      \multicolumn{3}{|c|}{Connected Component I}\\
      \hline
      Type & \# & Degree\\
      \hline
      $\PP^1\times\PP^1$ &$9$&$8$\\
      \hline
    \end{tabular}
&\quad
\begin{tabular}{|c| c| c|}
     \hline
      \multicolumn{3}{|c|}{Connected Components II and III}\\
      \hline
      Type & \# & Degree\\
      \hline
      $\PP^5$&$3$&$1$\\
      $\F_1$&$9$&$8$\\
      Fat point & $18$& $-$\\
      \hline
    \end{tabular}

  \end{tabular}
\vspace{.5cm}

    \caption{Primary decomposition for the connected components of $\bF_4(P_{3,3}^3)$.}\label{table:compsp}
  \end{table}

  \begin{prop}\label{prop:f4p333}
    The Fano scheme $\bF_4(P_{3,3}^3)$ has exactly three connected components. Their primary decompositions are as described in Table \ref{table:compsp}.
  \end{prop}

\section{Comparisons and Further Questions}\label{sec:conclusion}

In our study of Fano schemes of determinants and permanents, we have seen that a good understanding of the components $\comp_k(s)$ and $\comp_k(\sigma,\tau)$ induced by compression spaces can already say a lot about the geometry of the entire Fano scheme.  One striking difference between $\bF_k(\dmn)$ and $\bF_k(\pmn)$ is that while the components $\comp(s)$ are positive-dimensional, the components $\comp(\sigma,\tau)$ are isolated points. This has the following consequences:
\begin{itemize}
  \item For large values of $k$, the dimension of $\bF_k(\dmn)$ is greater than that of $\bF_k(\pmn)$. In other words, $\dmn$ contains `more' high-dimensional linear spaces than $\pmn$.
  \item For large values of $k$, $\bF_k(\dmn)$ has fewer irreducible components than $\bF_k(\pmn)$.
\end{itemize}

For the case of lines on the determinantal and permanental hypersurfaces, however, we conjecture that $P_{n,n}^n$ behaves exactly as $D_{n,n}^n$ does; we have verified that this is the case for $n=3,4$.

\begin{conj}\label{conj:expected}
  The Fano scheme $\bF_1(P_{n,n}^n)$ has dimension $2n^2-n-5$, that is, the expected dimension, and is a reduced local complete intersection.
\end{conj}

Even though the compression subschemes $\comp_k(s)$ and $\comp(\sigma,\tau)$ give a surprising amount of information about the global structure of our Fano schemes, one way in which they fail to tell the whole story is that they don't by themselves detect connectedness, as far as we know.  See Remark~\ref{rem:connectedness} for a detailed discussion.  In particular, we would like to know:

\begin{question}
Is the Fano scheme $\bF_{10}(D^4_{5,5})$ connected? (See Remark~\ref{rem:connectedness}.)  
More generally, is there a Fano scheme $\fd$ which is connected but whose compression components $\comp_k(s)$ are not by themselves connected?

\end{question}

A frequently asked question regarding a moduli space is whether it is \emph{rational} or \emph{unirational}. The components $\comp_k(s)$ of $\bF_k(\dmn)$ are by construction unirational; furthermore, the single exotic component $\comp^*$ of $\bF_2(D_{3,3}^3)$ is also obviously unirational. 
We suspect that all irreducible components of $\bF_k(\dmn)$ may well be unirational. Rationality is more delicate: we do not even know if the components $\comp_k(s)$ are rational, although we can prove that they are for the cases $s=0$ and $r-1$.
\begin{question}
  Are the irreducible components of $\bF_k(\dmn)$ unirational? Are they rational?
\end{question}

Finally, our choice of the schemes $\dmn$ and $\pmn$ was motivated by Valiant's conjecture.  They can be simultaneously generalized, however, in the following interesting way.  Fix any character $\chi$ of the symmetric group $S_r$ on $r$ letters.  The {\em $\chi$-immanant} of an $r\times r$ matrix $(x_{ij})$ is the degree $r$ polynomial with $r!$ terms
$$\textrm{Imm}_\chi = \sum_{\sigma\in S_r} \chi(\sigma) x_{1\sigma(1)}\cdots x_{r\sigma(r)}.$$
Thus the $\chi$-immanant specializes to the determinant and permanent by choosing $\chi$ to be the sign or trivial character, respectively.  Now consider the scheme $D^{\chi,r}_{m,n} \subset \PP^{mn-1}$ cut out by the $r\times r$ $\chi$-immanants of a general $m\times n$ matrix.  As $\chi$ varies, we have a family of schemes that interpolate between $\dmn$ and $\pmn$.
\begin{question}
  What is the structure of the Fano scheme $\bF_k(D^{\chi,r}_{m,n})$?
Does $\bF_k(D^{\chi,r}_{m,n})$ behave like $\bF_k(\dmn)$ or $\bF_k(\pmn)$?  
\end{question}

\bibliographystyle{alpha}
\bibliography{fano}

\end{document}